\documentclass[tbtags,10pt,a4paper]{amsart}

\numberwithin{equation}{section}
\allowdisplaybreaks % let equations break across pages (doesn't work with groups like aligned)

\usepackage{amssymb,eucal,mathrsfs,setspace,bbm}
\usepackage[noadjust,nobreak]{cite}
\usepackage[inner=25mm, outer=25mm, top=24mm, bottom=24mm, head=10mm, foot=10mm]{geometry}

\usepackage{array}
\newcolumntype{C}{>{$}c<{$}} %Defines math mode in tabular

\usepackage[largesc,theoremfont]{newtxtext}      % mathptmx is now considered obsolete...
\usepackage[libertine,cmbraces,varbb,smallerops]{newtxmath}
\begingroup
  \makeatletter
  \@for\theoremstyle:=definition,remark,plain\do{%
    \expandafter\g@addto@macro\csname th@\theoremstyle\endcsname{%
      \addtolength\thm@preskip{.5\baselineskip plus .2\baselineskip minus .2\baselineskip}
      \addtolength\thm@postskip{.5\baselineskip plus .2\baselineskip minus .2\baselineskip}
    }%
  }
\endgroup
\usepackage{enumitem}
\setitemize{leftmargin=*}   % Removes margin from itemized lists (enumitem package)
\setenumerate{leftmargin=*, % Removes margin for enumerate lists (enumitem package)
  label=\textup{(\arabic*)}%, % Makes labels upright numbers
	}

\usepackage{mathtools} % needed for cases* and the bra/ket constructions; loads amsmath
\usepackage{graphicx}

\usepackage{tikz}
\tikzset{%
	>=latex,
	wt/.style={circle, draw=black, fill=black, inner sep=2pt, outer sep=0pt, minimum size=5pt}, % a weight
	rwt/.style={circle, draw=red, fill=red, inner sep=2pt, outer sep=0pt, minimum size=5pt}, % red variant
}

\usepackage[colorlinks=true,citecolor=red,linkcolor=blue,linktoc=all]{hyperref} % load _before_ cleveref!

\usepackage[capitalise,noabbrev]{cleveref} % needed for \cref and \Cref (almost all refs except \eqref)

\theoremstyle{plain}
\newtheorem{theorem}{Theorem}[section]
\newtheorem{corollary}[theorem]{Corollary}
\newtheorem{lemma}[theorem]{Lemma}
\newtheorem{proposition}[theorem]{Proposition}

\newtheorem*{conjecture}{Conjecture}
\newtheorem{definition}[theorem]{Definition}
\newtheorem{remark}[theorem]{Remark}

\newcommand{\pd}{\partial}
\newcommand{\wun}{\mathbbm{1}}  % identity field/operator (requires newtxmath)

\newcommand{\cc}{\mathsf{c}}   % central charge
\newcommand{\dd}{\mathrm{d}}   % d in derivatives and integrals
\newcommand{\ee}{\mathrm{e}}   % ln e = 1
\newcommand{\ii}{\mathrm{i}}   % imaginary unit
\newcommand{\kk}{\mathsf{k}}   % level
\newcommand{\uu}{\mathsf{u}}   % minmod parameter
\providecommand{\vv}{}\renewcommand{\vv}{\mathsf{v}} % minmod parameter (newtxmath predefines \vv to be something else)

\renewcommand{\ge}{\geqslant} % never use \geq or \geqslant, use \ge which is globally defined to be one or the other
\renewcommand{\le}{\leqslant} % same for \leq and \leqslant (and we should use \ne instead of \neq for consistency I guess)

\DeclarePairedDelimiter{\brac}{\lparen}{\rparen}   % use \brac for (...) and \brac* to automatically scale the ( and )
\DeclarePairedDelimiter{\sqbrac}{\lbrack}{\rbrack} % use \sqbrac[\big] for \bigl(...\bigr) etc...
\DeclarePairedDelimiter{\set}{\lbrace}{\rbrace}
\newcommand{\st}{\mspace{5mu} {:} \mspace{5mu}}    % "such that" in sets
\DeclarePairedDelimiter{\abs}{\lvert}{\rvert}

\newcommand{\bilin}[2]{\left\langle #1 , #2 \right\rangle}
\DeclarePairedDelimiterX{\comm}[2]{\lbrack}{\rbrack}{#1 , #2}  % commutators
\newcommand{\no}[1]{\mathopen{:} #1 \mathclose{:}} % normal ordering (prevent := or =:)

\DeclareMathOperator{\spn}{span}

\newcommand{\Ra}{\Rightarrow}
\newcommand{\Lra}{\Leftrightarrow}
\newcommand{\lra}{\longrightarrow}
\newcommand{\ira}{\hookrightarrow}

\newcommand{\ses}[3]{0 \lra #1 \lra #2 \lra #3 \lra 0}

\newcommand{\fld}[1]{\mathbb{#1}}    % for fields and related things
\newcommand{\alg}[1]{\mathfrak{#1}}  % for Lie algebras
\newcommand{\grp}[1]{\mathsf{#1}}    % for groups
\newcommand{\VOA}[1]{\mathsf{#1}}    % for VOAs
\newcommand{\Mod}[1]{\mathcal{#1}}   % for modules
\newcommand{\categ}[1]{\mathscr{#1}} % for categories

\newcommand{\ZZ}{\fld{Z}}
\newcommand{\NN}{\ZZ_{\ge 0}} %\fld{N}} %D modified to eliminate confusion among uneducated people

\newcommand{\RR}{\fld{R}}
\newcommand{\CC}{\fld{C}}

\newcommand{\affine}[1]{\widehat{#1}}
\newcommand{\SLA}[2]{\alg{#1}_{#2}}                      % Lie algebras like sl(2)
\newcommand{\SLSA}[3]{\alg{#1} (#2 \vert #3)}            % Lie superalgebras like gl(1|1)
\newcommand{\AKMA}[2]{\affine{\alg{#1}}_{#2}}            % Kac-Moody algebras
\newcommand{\sltwo}{\SLA{sl}{2}}

\newcommand{\slthree}{\SLA{sl}{3}}
\newcommand{\aslthree}{\AKMA{sl}{3}}

\newcommand{\fwt}[1]{\omega_{#1}}                        % fundamental weight
\newcommand{\wlat}{\grp{P}}                              % weight lattice
\newcommand{\pwlat}[1]{\wlat_{\ge}^{#1}}                 % dominant affine alcove

\DeclarePairedDelimiterXPP{\uealg}[1]{\mathsf{U}}{\lparen}{\rparen}{}{#1}   % universal enveloping algebra

\newcommand{\uaff}[2]{\VOA{V}^{#1}(#2)}            % universal affine VA
\newcommand{\saff}[2]{\VOA{L}_{#1}(#2)}            % simple affine VA
\newcommand{\usl}[1]{\uaff{#1}{\slthree}}          % universal affine sl_3
\newcommand{\ssl}[1]{\saff{#1}{\slthree}}          % simple affine sl_3
\newcommand{\uslk}{\usl{\kk}}                      % universal affine sl_3
\newcommand{\sslk}{\ssl{\kk}}                      % simple affine sl_3

\newcommand{\bpsymb}{\VOA{BP}}
\newcommand{\zamsymb}{\VOA{W}_3}
\newcommand{\slthreesymb}{\VOA{A}_2}

\newcommand{\ubp}[1]{\bpsymb^{#1}}             % universal BP
\newcommand{\sbp}[1]{\bpsymb_{#1}}             % simple BP
\newcommand{\bpminmod}[2]{\bpsymb(#1,#2)}      % BP minimal model
\newcommand{\uzam}[1]{\zamsymb^{#1}}           % universal W_3
\newcommand{\wminmod}[2]{\zamsymb(#1,#2)}      % W_3 minimal model
\newcommand{\slminmod}[2]{\slthreesymb(#1,#2)} % sl_3 minimal model

\newcommand{\ubpk}{\ubp{\kk}}
\newcommand{\sbpk}{\sbp{\kk}}
\newcommand{\bpminmoduv}{\bpminmod{\uu}{\vv}}
\newcommand{\uzamk}{\uzam{\kk}}
\newcommand{\szamk}{\VOA{W}_{3,\kk}}
\newcommand{\wminmoduv}{\wminmod{\uu}{\vv}}
\newcommand{\slminmoduv}{\slminmod{\uu}{\vv}}

\newcommand{\zammaxprop}[1]{\VOA{J}^{#1}}   % maximal ideal of universal W3
\newcommand{\zammaxpropk}{\zammaxprop{\kk}}
\newcommand{\bpmaxprop}[1]{\VOA{K}^{#1}}    % maximal ideal of universal BP
\newcommand{\bpmaxpropk}{\bpmaxprop{\kk}}

\newcommand{\heis}{\VOA{H}}             % Heisenberg VOA
\newcommand{\lsymb}{\Pi}
\newcommand{\lvoa}{\lsymb}              % the half-lattice vertex algebra

\newcommand{\cclvoa}{\cc^{\lsymb}_{\kk}}  % lattice VOA central charge
\newcommand{\cczam}{\cc^{\zamsymb}_{\kk}} % W_3 VOA central charge
\newcommand{\ccbp}{\cc^{\bpsymb}_{\kk}}   % BP VOA central charge

\newcommand{\bpconj}{\gamma}            % BP conjugation
\newcommand{\bpsfsymb}{\sigma}          % BP spectral flow
\newcommand{\lsfsymb}{\varsigma}        % Pi spectral flow

\newcommand{\bpsforbit}[1]{\mathbb{O}_{#1}} % BP spectral flow orbit

\DeclarePairedDelimiterXPP{\zhu}[1]{\mathsf{Zhu}}{\lbrack}{\rbrack}{}{#1}

\newcommand{\fock}[1]{\Mod{F}_{#1}}           % bosonic Fock space
\newcommand{\sing}[1]{\Mod{S}_{#1}}           % irreducible singlet module
\newcommand{\lmod}[1]{\Pi_{[#1]}}             % relaxed hw lattice module
\newcommand{\wmod}[2]{\Mod{W}_{#1,#2}}        % irreducible W3 module
\newcommand{\bpmod}[2]{\Mod{H}_{#1,#2}}       % irreducible hw BP module
\newcommand{\bpcmod}[2]{\Mod{C}_{#1,#2}}      % irreducible conjugate hw BP module
\newcommand{\bprmod}[3]{\Mod{R}_{[#1],#2,#3}} % relaxed BP module (#1 = lattice charge, #2,#3 = W3 labels)
\newcommand{\slmod}[2]{\Mod{L}_{#1,#2}}       % irreducible hw asl_3 module

\newcommand{\wmodhw}{\wmod{h}{w}}
\newcommand{\bpmodjD}{\bpmod{j}{\Delta}}
\newcommand{\bpcmodjD}{\bpcmod{j}{\Delta}}
\newcommand{\bprmodjhw}{\bprmod{j}{h}{w}}

\newcommand{\irreps}[2]{I_{#1,#2}}            % parameters giving irreducible hwms for a W_3 minimal model
\newcommand{\irrepsuv}{\irreps{\uu}{\vv}}

\newcommand{\wcat}{\categ{W}}  % weight category
\newcommand{\wcatk}{\wcat^{\kk}}
\newcommand{\wcatuv}{\wcat_{\uu,\vv}}

\newcommand{\klcat}{\categ{K\!L}} % Kazhdan--Lusztig category

\newcommand{\qhrfunc}[1]{\Phi^{\textup{#1}}} % quantum hamiltonian reduction functors
\newcommand{\qhrmin}{\qhrfunc{min.}}         % minimal
\DeclareMathOperator{\tr}{tr}
\newcommand{\Gr}[1]{\bigl[#1\bigr]}                    % element of a Grothendieck group/ring
\newcommand{\traceover}[1]{\tr_{\raisebox{-2pt}{$\scriptstyle #1$}}} % for ch[M] = tr_M ...

\DeclareMathOperator{\chmap}{ch}
\newcommand{\ch}[1]{\chmap \Gr{#1}}                    % character
\newcommand{\fch}[2]{\ch{#1} \left(#2\right)}          % new character as function of q and ...

\newcommand{\bp}{Bershadsky--Polyakov}
\newcommand{\fs}{Feigin--Semikhatov}
\newcommand{\km}{Kac--Moody}
\newcommand{\kl}{Kazhdan--Lusztig}
\newcommand{\pbw}{Poincar\'{e}--Birkhoff--Witt}
\newcommand{\zam}{Zamolodchikov}

\newcommand{\fdim}{finite-dimensional}
\newcommand{\infdim}{in\fdim}

\newcommand{\rhs}{right-hand side}

\newcommand{\hw}{highest-weight}
\newcommand{\hwv}{\hw\ vector}
\newcommand{\hwvs}{\hwv s}
\newcommand{\hwm}{\hw\ module}
\newcommand{\hwms}{\hwm s}

\newcommand{\rhw}{relaxed highest-weight}
\newcommand{\rhwv}{\rhw\ vector}
\newcommand{\rhwvs}{\rhwv s}
\newcommand{\rhwm}{\rhw\ module}
\newcommand{\rhwms}{\rhwm s}
\newcommand{\fr}{fully relaxed}
\newcommand{\frm}{\fr\ module}
\newcommand{\frms}{\frm s}
\newcommand{\chw}{conjugate highest-weight}
\newcommand{\chwv}{\chw\ vector}
\newcommand{\chwvs}{\chwv s}
\newcommand{\chwm}{\chw\ module}
\newcommand{\chwms}{\chwm s}

\newcommand{\vo}{vertex operator}
\newcommand{\voa}{\vo\ algebra}
\newcommand{\voas}{\voa s}

\newcommand{\va}{vertex algebra}

\newcommand{\vsa}{vertex subalgebra}

\newcommand{\ope}{operator product expansion}
\newcommand{\opes}{\ope s}
\newcommand{\qhr}{quantum hamiltonian reduction} %DR yes the _h_ is lowercase
\newcommand{\qhrs}{\qhr s}
\newcommand{\uea}{universal enveloping algebra}

\makeatletter
\renewcommand\author@andify{%
  \nxandlist {\unskip ,\penalty-1 \space\ignorespaces}%
    {\unskip {} \@@and~}%
    {\unskip \penalty-2 \space \@@and~}%
}
\makeatother

\DeclareRobustCommand{\SkipTocEntry}[5]{}

\usepackage{xpatch}
\makeatletter
\xpatchcmd{\@tocline}
          {\hfil\hbox to\@pnumwidth{\@tocpagenum{#7}}\par}
          {\ifnum#1<0\hfill\else\dotfill\fi\hbox to\@pnumwidth{\@tocpagenum{#7}}\par}
          {}{}
\makeatother

\begin{document}

\title{Weight module classifications for Bershadsky--Polyakov algebras}

\author[D~Adamovi\' c]{Dra\v zen Adamovi\' c}
\address[Dra\v zen Adamovi\' c]{Department of Mathematics, Faculty of Science \\
University of Zagreb \\
Bijeni\v cka 30, Croatia.
}
\email{adamovic@math.hr}

\author[K~Kawasetsu]{Kazuya Kawasetsu}
\address[Kazuya Kawasetsu]{
Priority Organization for Innovation and Excellence \\
Kumamoto University \\
Kumamoto 860-8555, Japan.
}
\email{kawasetsu@kumamoto-u.ac.jp}

\author[D~Ridout]{David Ridout}
\address[David Ridout]{
School of Mathematics and Statistics \\
University of Melbourne \\
Parkville, Australia, 3010.
}
\email{david.ridout@unimelb.edu.au}

\begin{abstract}
	The \bp\ algebras are the subregular \qhrs\ of the affine \voas\ associated with $\slthree$.
	In \cite{AdaRea20}, we realised these algebras in terms of the regular reduction, \zam's W$_3$-algebra, and an isotropic lattice \voa.
	We also proved that a natural construction of \rhw\ \bp\ modules has the property that the result is generically irreducible.
	Here, we prove that this construction, when combined with spectral flow twists, gives a complete set of irreducible weight modules whose weight spaces are \fdim.
	This gives a simple independent proof of the main classification theorem of \cite{FehCla20} for nondegenerate admissible levels and extends this classification to a category of weight modules.
	We also deduce the classification for the nonadmissible level $\kk=-\frac{7}{3}$, which is new.
\end{abstract}

\maketitle

\markleft{D~ADAMOVI\'{C}, K~KAWASETSU AND D~RIDOUT} % don't know how to remove the oxford comma in the running head

\tableofcontents

\onehalfspacing

\section{Introduction}

\subsection{Background}

Among the most important \voas\ are the affine ones.
As one might expect, the members of this family that are associated with $\sltwo$ are the most tractable.
In this case, one can distinguish the universal \voa\ $\uaff{\kk}{\sltwo}$, where $\kk \in \CC \setminus \set{-2}$ denotes the level, from its simple quotient $\saff{\kk}{\sltwo}$.
In fact, these are distinct if and only if $\kk$ is admissible, a technical condition introduced in \cite{KacMod88}.

The best understood $\saff{\kk}{\sltwo}$ are those with $\kk \in \NN$.
For this subset of admissible levels, $\saff{\kk}{\sltwo}$ is strongly rational \cite{WitNon84,GepStr86,FreVer92}.
The remaining admissible levels are perhaps even more interesting because then $\saff{\kk}{\sltwo}$ admits finitely many irreducible \hwms\ but an uncountably infinite number of other irreducible modules \cite{AdaVer95}.
Moreover, the characters of the \hwms\ span a representation of the modular group (this was the motivation for the introduction of admissibility in \cite{KacMod88}).
Unfortunately, for admissible levels that are not nonnegative integers, Verlinde's formula \cite{VerFus88} for the fusion multiplicities fails \cite{KohFus88}.

It took twenty years to properly understand the reason behind this failure \cite{RidSL208} and another five to fix it \cite{CreMod13}.
The modern approach to the representation theory of $\saff{\kk}{\sltwo}$ at general admissible levels prioritises the so-called \rhwms, named in \cite{FeiEqu97} but previously classified in \cite{AdaVer95}, and their images under twisting by spectral flow automorphisms.
It is the characters of these modules that carry the true representation of the modular group, consistent with (a mild generalisation of) Verlinde's formula \cite{CreLog13,RidVer14}.

The characters of the \rhw\ $\saff{\kk}{\sltwo}$-modules (and their spectral flows) were proposed in \cite{CreMod12,CreMod13} and proven in \cite{KawRel18}.
Interestingly, they turn out to be proportional to the characters of the irreducible \hwms\ of a Virasoro minimal model \voa.
And not just any minimal model, but the \qhr\ of $\saff{\kk}{\sltwo}$.
This beautiful observation demanded a beautiful explanation and one was subsequently provided in \cite{AdaRea17} through a functorial construction that we call (following \cite{SemInv94}) inverse \qhr.

This construction realises a \rhw\ $\saff{\kk}{\sltwo}$-module as a tensor product of a \hw\ Virasoro module with a module over a specific lattice \voa\ $\lvoa$.
It has since been generalised to several other affine \voas\ and W-algebras, including $\saff{\kk}{\SLSA{osp}{1}{2}}$ and $N=1$ super-Virasoro \cite{AdaRea17}, the \bp\ and \zam\ algebras \cite{AdaRea20}, $\sslk$ and \bp\ \cite{AdaRel21}, and the \fs\ and W$_n$ Casimir algebras \cite{FehSub21}.
The general philosophy here is that the representation theory of a given nonrational affine \voa\ (or W-algebra), which is relatively complicated, should be reconstructed using inverse \qhr\ functors from that of another less complicated (perhaps even rational) W-algebra.

\subsection{The state of the art}

It is natural when faced with an algebraic structure to first try to classify its irreducible modules in an appropriate category.
In our case, the algebraic structure is an affine \voa\ or one of its W-algebras and an appropriate category is that consisting of the weight modules with \fdim\ weight spaces.
(This latter condition is reasonable if one wishes to study characters and their modular properties.)
The corresponding classifications are known for certain (nonsuper) rational \voas\ including affine ones at nonnegative-integer levels \cite{FreVer92}, regular W-algebras at nondegenerate admissible levels \cite{AraRat12b} and (more generally) certain W-algebras said to be exceptional \cite{AraRat19,McRRat21}.

The situation for nonrational affine \voas\ and W-algebras is not as satisfactory.
As noted above, the classification for $\saff{\kk}{\sltwo}$, with $\kk$ admissible, was completed in \cite{AdaVer95} but only for the full subcategory of \rhwms.
More recently, similar classifications have appeared for $\saff{\kk}{\SLSA{osp}{1}{2}}$ \cite{CreCos18}, $\saff{\kk}{\slthree}$ \cite{AraWei16,KawAdm21} and the simple minimal W-algebras associated to $\slthree$ \cite{FehCla20} and $\SLSA{sl}{2}{1}$ \cite{CreUni19}.
Unfortunately, the methods used in these works appear to be difficult to generalise.

If we further restrict to the full subcategory of \hwms, or more precisely the vertex-algebraic analogue of the BGG category $\categ{O}_{\kk}$, then the classification was established for all nonsuper affine \voas\ when the level is admissible \cite{AraRat12}.
The corresponding relaxed classification was subsequently shown to follow algorithmically in \cite{KawRel19}.
However, it seems that even the \hw\ classification remains out of reach for general W-algebras (and almost all superalgebras).

Our thesis is that the inverse \qhr\ functors of \cite{AdaRea17} provide a powerful new way to classify irreducible \rhwms\ of nonrational affine \voas\ and W-algebras.
By this, we mean that we expect that applying these functors to irreducible modules will result in generically irreducible \rhwms\ and that all irreducible \rhwms\ may be constructed in this manner.
(We add the qualifier ``generically'' here as some of the \rhwms\ constructed by inverse reduction are necessarily reducible.)

These expectations were shown to be met for $\saff{\kk}{\sltwo}$ in \cite{AdaRea17} by applying inverse \qhr\ functors to irreducible Virasoro modules and comparing with the known character formulae and relaxed classification for $\saff{\kk}{\sltwo}$.
As the latter results are not available for comparison in general, it becomes desirable to develop proofs that instead rely principally on inverse \qhr.
In \cite{AdaRea20}, we satisfied a part of this desire by constructing an intrinsic proof that inverse reduction maps irreducible modules to generically irreducible ones.
This was presented for the simple \bp\ \voas\ $\sbpk$ of nondegenerate admissible levels $\kk$, rather than for $\saff{\kk}{\sltwo}$, in order to illustrate the method in a nonaffine example.
(The modifications required for $\saff{\kk}{\sltwo}$ are very simple and were left to the reader.)
The generality of our method was subsequently confirmed in \cite{FehSub21}, where this generic irreducibility was established for the subregular W-algebras associated to $\SLA{sl}{n}$.

It remains to develop an intrinsic means to prove that inverse \qhr\ constructs all irreducible \rhwms, up to isomorphism.
This is the task we set ourselves in this paper.
We shall again present the method for $\sbpk$, noting that it may be readily adapted for $\saff{\kk}{\sltwo}$.
The expectation is that it will also readily generalise to higher-rank cases.

\subsection{Results}

Recall from \cite[Thms.~3.6 and 6.2]{AdaRea20} that inverse \qhr\ functors are defined for $\sbpk$ if and only if $\kk \notin \set{-3} \cup \frac{1}{2} \ZZ_{\ge-3}$.
The main results below all assume this restriction on the level.
Let $\szamk$ denote the simple regular W-algebra of level $\kk$ associated with $\slthree$.
Our first main result is then as follows:
\begin{enumerate}[label=(M\arabic*)]
	\item\label{it:mainthm1} Every irreducible fully \rhw\ $\sbpk$-module is isomorphic to the result of applying some inverse \qhr\ functor to some irreducible \hw\ $\szamk$-module.
\end{enumerate}
Here, we use the term ``fully relaxed'' to exclude the irreducible \hw\ and \chw\ modules that cannot be so realised (see \cref{def:topfullyrelaxed}).
However, these irreducibles are easily brought into the fold because of our second main result:
\begin{enumerate}[resume*]
	\item\label{it:mainthm2} Every irreducible \hw\ or \chw\ $\sbpk$-module is isomorphic to a spectral flow image of a quotient of a reducible fully \rhw\ $\sbpk$-module constructed as in \ref{it:mainthm1}.
\end{enumerate}
In fact, we may equivalently replace ``quotient'' by ``submodule'' in this result.

These two results complete the classification of irreducible $\sbpk$-modules in the relaxed category.
When $\kk$ is nondegenerate admissible, this reproduces the main result of \cite{FehCla20}.
Their proof relies heavily on the special properties of the minimal \qhr\ functor \cite{KacQua03,KacQua03b,AraRep04} and is therefore difficult to generalise to other nilpotent orbits.
Our proof does not have this problem as the \qhr\ functor we use is the regular one, needed only to classify the irreducible \hwms\ of $\szamk$, and this classification is known for higher ranks \cite{AraRat12}.
We expect that our methods will also generalise to degenerate admissible levels using the theory of exceptional W-algebras recently developed in \cite{AraRat19}.

Our inverse reduction methods also apply to nonadmissible levels of the form $\kk = -3+\frac{2}{\vv}$, where $\vv\ge3$ is odd.
For these levels, the classification given by our two main results is new.
When $\vv=3$, hence $\kk = -\frac{7}{3}$, we can make this classification explicit because $\szamk$ then coincides with the singlet algebra \cite{KauExt91} of central charge $\cc=-2$ whose representation theory is well understood, see \cite{EhoRep92,HonAut92,WanCla97,AdaCla02,CreRib20,CreRib22}.
When $\vv>3$, it remains an open problem to make the classification explicit.

Nevertheless, the $\kk = -\frac{7}{3}$ results are very interesting.
Whereas for nondegenerate levels, one obtains a finite number of \hwms, here we have four one-parameter families of such modules, one of which consists entirely of ordinary modules.
Correspondingly, we have a two-parameter family of generically irreducible \rhwms, contrasting with the one-parameter result for nondegenerate levels.
In a sense, this combines the nondegenerate result with that obtained for the nonadmissible levels $\kk\in\ZZ_{\ge-1}$ in \cite{AdaCla19,AdaBer20}.
For these levels, our inverse reduction methods do not apply, but singular vector methods may be used to deduce the existence of one-parameter families of \hwms, all of which are ordinary, and no (fully) relaxed families.

The classification of irreducible $\sbpk$-modules is therefore now very well understood in the relaxed category.
However, we are ultimately interested in the larger category of weight $\sbpk$-modules with \fdim\ weight spaces.
Happily, the classification in this category is covered by our third main result:
\begin{enumerate}[resume*]
	\item\label{it:mainthm3} Every irreducible weight $\sbpk$-module, with \fdim\ weight spaces, is isomorphic to a spectral flow of either a fully \rhwm\ or a \hwm.
\end{enumerate}
As far as we can tell, this result is also new, as is the corresponding result for $\saff{\kk}{\sltwo}$ (which is easily obtained using the same methods).
Using our method, we can also prove that for $\kk\in\ZZ_{\ge-1}$, the irreducible positive-energy $\sbpk$-modules uncovered in \cite{AdaBer20} likewise give all the irreducible weight $\sbpk$-modules with \fdim\ weight spaces.

\subsection{Outline}

We commence in \cref{sec:bpreview} by reviewing the theory of inverse \qhr\ functors between \zam\ and \bp\ modules, following \cite{AdaRea20}.
The discussion also serves to fix our notation and conventions.
The work begins in \cref{sec:heis,sec:wtcat}.
We first adapt some seminal results of Futorny \cite{FutIrr96} to the rank-$1$ Heisenberg vertex algebra.
These allow us to prove our main result \ref{it:mainthm3} above, see \cref{prop:irred=sf+relax}.

We return to inverse \qhr\ in \cref{sec:complete}.
It is not difficult to see that these functors produce every \rhwm\ for the universal \bp\ algebra $\ubpk$ (\cref{prop:completeuniv}).
The extension, \cref{prop:completeminmod}, to $\sbpk$, $\kk \notin \set{-3} \cup \frac{1}{2} \ZZ_{\ge-3}$, is our main result \ref{it:mainthm1}.
It requires a technical lemma, which we prove using the string function methods developed in \cite[App.~A]{KawRel18}, and a comparison of the maximal ideal of $\ubpk$ and that of the universal regular W-algebra $\uzamk$.

\cref{sec:completehw} then addresses the irreducible \hwms, noting first (\cref{prop:infdimhwm}) that such a module may always be realised as a quotient of a reducible \rhwm\ if its subspace of minimal conformal weight (equivalently, its image under the Zhu functor) is \infdim.
We then prove that the remaining irreducible \hwms\ can be obtained from these quotients using spectral flow (\cref{prop:fdimhwm}), thereby establishing our main result \ref{it:mainthm2}.
(This proof is the only place in which we need to use the explicit form of the embedding that underlies the inverse reduction functors.  It would be nice to dispense with it entirely, assuming that this is possible.)

As a first application of these general results, the classification of irreducible weight modules for nondegenerate levels is quickly described in \cref{sec:nondeg}.
The analogous (but new) classification for $\kk=-\frac{7}{3}$ is then detailed in \cref{thm:bp(23)class}.
We also extract from this \lcnamecref{thm:bp(23)class} the classification of irreducible ordinary $\sbp{-7/3}$-modules (\cref{thm:bp(23)fdtopclass}).
We conclude in \cref{sec:sl3} by proving a few simple consequences of our results for the irreducible ordinary modules of $\sslk$, $\kk \notin \set{-3} \cup \frac{1}{2} \ZZ_{\ge-3}$.
In particular, we deduce another new result (\cref{cor:a23kl}): the classification of irreducible ordinary modules for $\sslk$ at the nonadmissible level $\kk=-\frac{7}{3}$.

Finally, let us recall a general principle/conjecture of vertex algebra theory (and conformal field theory) which says that every irreducible module for a vertex subalgebra $\VOA{U}$ of a vertex algebra $\VOA{V}$ may be obtained from $\VOA{V}$-modules or twisted $V$-modules.
Here, we test this principle when $\VOA{U} = \ubpk$ and $\VOA{V}= \uzamk \otimes \lvoa$ in the category of weight modules.
We expect that this should be also verified for general affine vertex algebras and W-algebras related by inverse \qhr.

\addtocontents{toc}{\SkipTocEntry}
\subsection*{Acknowledgements}

We thank Thomas Creutzig, Justine Fasquel, Zac Fehily, Slava Futorny, Chris Raymond and Simon Wood for discussions related to the material presented here.

D.A.\ is partially supported by the
QuantiXLie Centre of Excellence, a project cofinanced
by the Croatian Government and European Union
through the European Regional Development Fund - the
Competitiveness and Cohesion Operational Programme
(KK.01.1.1.01.0004).

KK's research is partially supported by
MEXT Japan ``Leading Initiative for Excellent Young Researchers (LEADER)'',
JSPS Kakenhi Grant numbers 19KK0065, 21K13775 and 21H04993.

DR's research is supported by the Australian Research Council Discovery Project DP210101502 and an Australian Research Council Future Fellowship FT200100431.

\section{Realising \bp\ algebras and modules} \label{sec:bpreview}

In this \lcnamecref{sec:bp}, we review the relationship \cite{AdaRea20} between the regular and subregular W-algebras associated to $\slthree$, also known as the \zam\ \cite{ZamInf85} and \bp\ \cite{PolGau90,BerCon91} algebras, respectively.
We also review an explicit construction \cite{AdaRea20} of the \rhwms\ of the latter from those of the former.

Throughout, we shall find it convenient to parametrise our algebras by a level $\kk\ne-3$, a complex number that is ultimately identified as the eigenvalue of the central element of $\aslthree$ on the associated affine vertex algebra.
Our primary focus will be rational levels with $\kk+3>0$ for which we write
\begin{equation} \label{eq:kuv}
	\kk+3 = \frac{\uu}{\vv}, \quad \text{where}\ \uu, \vv \in \ZZ_{>0}\ \text{and}\ \gcd \set{\uu,\vv} = 1.
\end{equation}
A level $\kk$ is said to be admissible if $\uu \ge 3$ and nondegenerate if, in addition, $\vv \ge 3$.

\subsection{A lattice \voa} \label{sec:latt}

We start with the ``half-lattice'' \voa\ $\lvoa$, studied in \cite{BerRep01} (see also \cite{FriCon86}).
Here, and throughout, let $\wun$ denote the identity field of a vertex algebra.
\begin{definition}
	Given $\kk \in \CC$, let $\lvoa$ denote the universal \voa\ with strong generators $c$, $d$ and $\ee^{nc}$, $n \in \ZZ$, subject to the following \opes
	\begin{equation} \label{eq:latopes}
		\begin{aligned}
			c(z) c(w) &\sim 0, & c(z) d(w) &\sim \frac{2 \, \wun}{(z-w)^2}, & d(z) d(w) &\sim 0, \\
			c(z) \ee^{nc}(w) &\sim 0, & d(z) \ee^{nc}(w) &\sim \frac{2n \, \ee^{nc}(w)}{z-w}, & \ee^{mc}(z) \ee^{nc}(w) &\sim 0,
			&&& m,n &\in \ZZ,
		\end{aligned}
	\end{equation}
	and equipped with the conformal vector
\begin{equation} \label{eq:defkappa}
	t = \tfrac{1}{2} \no{cd} + \kappa \pd c - \tfrac{1}{2} \pd d, \quad \kappa = \tfrac{1}{3} (2\kk+3).
\end{equation}
\end{definition}

This \voa\ is simple.
The conformal weights of the generators $c$, $d$ and $\ee^{nc}$ are $1$, $1$ and $n$, respectively, while the central charge is
\begin{equation}
	\cclvoa = 2+24\kappa.
\end{equation}
We therefore take the corresponding field expansions to be
\begin{equation}
	c(z) = \sum_{m \in \ZZ} c_m z^{-m-1}, \quad
	d(z) = \sum_{m \in \ZZ} d_m z^{-m-1} \quad \text{and} \quad
	\ee^{nc}(z) = \sum_{m \in \ZZ} \ee^{nc}_m z^{-m-n}.
\end{equation}
Note that the first three \opes\ of \eqref{eq:latopes} describe a symmetric bilinear form on $\spn \set{c,d}$ with $\bilin{c}{c} = \bilin{d}{d} = 0$ and $\bilin{c}{d} = 2$.
For later purposes, it will be convenient to introduce an alternative basis to $c$ and $d$, at least when $\kappa\ne0$, namely
\begin{equation} \label{eq:theotherbasis}
	a = \tfrac{1}{2} (d - \kappa c) \quad \text{and} \quad b = \tfrac{1}{2} (d + \kappa c).
\end{equation}

\begin{definition}
	\leavevmode
	\begin{itemize}
		\item The simultaneous eigenspaces of $c_0$ and $d_0$, acting on some $\lvoa$-module, are called weight spaces and their nonzero elements are weight vectors.
		\item A weight $\lvoa$-module is then a module that is the direct sum of its weight spaces.
		\item A \rhwv\ for $\lvoa$ is a weight vector that is annihilated by the $c_m$, $d_m$ and $\ee^{nc}_m$, $n \in \ZZ$, with $m>0$.
		\item A \rhw\ $\lvoa$-module is a module that is generated by a \rhwv.
	\end{itemize}
\end{definition}
\noindent We remark that a \rhwv\ is automatically an eigenvector for $t_0$.

The irreducible \rhw\ $\lvoa$-modules were classified in \cite{BerRep01}.
Let $\lmod{j}$, $[j] \in \CC/\ZZ$, denote the \rhw\ $\lvoa$-module generated by a \rhwv\ $\ee^{-b+(j+\kappa)c}$ on which the zero modes of the generating fields act as follows:
\begin{equation}
	c_0 \ee^{-b+(j+\kappa)c} = -\ee^{-b+(j+\kappa)c}, \quad
	d_0 \ee^{-b+(j+\kappa)c} = (2j+\kappa) \ee^{-b+(j+\kappa)c}, \quad
	\ee^{nc}_0 \ee^{-b+(j+\kappa)c} = \ee^{-b+(j+n+\kappa)c}.
\end{equation}
The conformal weight of $\ee^{-b+jc}$ is then $\kappa$.
Moreover, we have $\lmod{j} \cong \lmod{j+1}$, explaining the notation.
Finally, $\lmod{j}$ is irreducible and every irreducible \rhw\ $\lvoa$-module is isomorphic to some $\lmod{j}$.

There are also irreducible weight $\lvoa$-modules that are not \rhw.
Up to isomorphism, these may all be obtained by twisting the action of $\lvoa$ on some $\lmod{j}$ by spectral flow.
Let $Y_{\lsymb}$ denote the vertex map of $\lvoa$, so that $A(z) \equiv Y_{\lsymb}(A,z)$ for all $A \in \lvoa$.
Then, the action of the spectral flow map $\lsfsymb^{\ell}$, $\ell \in \ZZ$, on $\lvoa$ is given by \cite{LiPhy97}
\begin{equation} \label{eq:lattsf}
	\lsfsymb^{\ell} \brac[\big]{A(z)} = Y_{\lsymb} \brac[\big]{\Sigma(\ell b,z) A, z}, \quad \text{where}\
	\Sigma(\ell b,z) = z^{-\ell b_0} \prod_{n=1}^{\infty} \exp \brac*{\frac{(-1)^n}{n} \ell b_n z^{-n}}.
\end{equation}
There is also a similar spectral flow map given by replacing $b$ in \eqref{eq:lattsf} by $a$, but we shall not need it here.

The map $\lsfsymb^{\ell}$ may be naturally lifted to an invertible functor on the category of weight $\lvoa$-modules that is defined elementwise on objects, $v \in M \mapsto \lsfsymb^{\ell}(v) \in \lsfsymb^{\ell}(M)$, so that the action on the spectrally flowed module is given by
\begin{equation} \label{eq:lattsfmod}
	A(z) \lsfsymb^{\ell}(v) = \lsfsymb^{\ell} \brac[\big]{\lsfsymb^{-\ell} \brac[\big]{A(z)} v},
	\qquad A \in \lvoa.
\end{equation}
Every irreducible weight $\lvoa$-module is then isomorphic to some $\lsfsymb^{\ell}(\lmod{j})$ with $\ell \in \ZZ$ and $[j] \in \CC/\ZZ$.
In fact, it is easy to check that
\begin{equation} \label{eq:lattsfvec}
	\lsfsymb^{\ell}(\ee^{-b+(j+\kappa)c}) = \ee^{(\ell-1)b+(j+\kappa)c}.
\end{equation}
In particular, the vacuum state $\ee^0$ of $\lvoa$ belongs to the vacuum module $\lsfsymb(\lmod{-\kappa})$.

\subsection{The \zam\ algebra} \label{sec:zam}

The \zam\ algebra was introduced in \cite{ZamInf85}.
Its universal version $\uzamk$ coincides with the regular (or principal) level-$\kk$ W-algebra associated to $\slthree$.
\begin{definition}
	The universal \zam\ algebra $\uzamk$ is the \voa\ strongly generated by two elements $T$ and $W$, subject to the \opes
	\begin{equation}
		\begin{gathered}
	    T(z)T(w) \sim \frac{\cczam \wun}{2(z-w)^4} + \frac{2T(w)}{(z-w)^2} + \frac{\partial T(w)}{z-w}, \qquad
	    T(z)W(w) \sim \frac{3 W(w)}{(z-w)^2} + \frac{\partial W(w)}{z-w}, \\
	    W(z)W(w) \sim \frac{2 \Lambda(w)}{(z-w)^2} + \frac{\partial \Lambda(w)}{z-w}
				+ A_{\kk} \sqbrac*{\frac{\cczam \wun}{3(z-w)^6} + \frac{2 T(w)}{(z-w)^4} + \frac{\partial T(w)}{(z-w)^3}
				+ \frac{\frac{3}{10}\partial^2 T(w)}{(z-w)^2}  + \frac{\frac{1}{15}\partial^3 T(w)}{z-w}}.
		\end{gathered}
	\end{equation}
	Here, $\kk \in \CC \setminus \set{-3}$ is the level, $\Lambda$ denotes the quasiprimary field $\no{TT} - \frac{3}{10} \pd^2 T$,
	\begin{equation}
		\cczam = -\frac{2(3\kk+5)(4\kk+9)}{\kk+3} \quad \text{and} \quad
		A_{\kk} = -\frac{(3\kk+4)(5\kk+12)}{2(\kk+3)} = \frac{22+5\cczam}{16}.
	\end{equation}
\end{definition}

For certain levels, including all nondegenerate ones, the universal \zam\ algebra $\uzamk$ is not simple \cite{MizDet89,WatDet89}.
For these levels, its unique simple quotient, which we shall denote by $\wminmoduv$, is called a $\zamsymb$ minimal model \voa.
For nondegenerate levels, $\wminmoduv$ is rational and lisse \cite{AraAss10,AraRat12b}.
Moreover, we have $\wminmoduv = \wminmod{\vv}{\uu}$ and $\wminmod{3}{4} = \wminmod{4}{3} \cong \CC$.

Define modes $T_m$ and $W_m$, $m \in \ZZ$, by expanding the generating fields as
\begin{equation}
	T(z) = \sum_{m \in \ZZ} T_m z^{-m-2} \quad \text{and} \quad W(z) = \sum_{m \in \ZZ} W_m z^{-m-3}.
\end{equation}
\begin{definition}
	\leavevmode
	\begin{itemize}
		\item The eigenspaces of $T_0$, acting on a $\uzamk$-module, are the module's weight spaces and their nonzero elements are its weight vectors.
		\item A weight $\uzamk$-module is a module that is the direct sum of its weight spaces.
		\item A \hwv\ for $\uzamk$ is a simultaneous eigenvector of $T_0$ and $W_0$ that is annihilated by the $T_m$ and $W_m$ with $m>0$.
		\item A \hw\ $\uzamk$-module is a module that is generated by a \hwv.
	\end{itemize}
\end{definition}
\noindent
It may seem tempting to refine the definition of a weight vector/space to instead be a simultaneous eigenspace of $T_0$ and $W_0$.
However, there are natural examples that render this undesirable, see for instance \cite[Sec.~2.2.2]{BouW3A95}.
In particular, $W_0$ need not act semisimply on a \hw\ $\uzamk$-module, even though it does on the generating \hwv.
With the above definitions, a \hw\ $\uzamk$-module is always a weight module.

An irreducible \hw\ $\uzamk$-module $\wmodhw$ is thus determined, up to isomorphism, by the eigenvalues $h$ of $T_0$ and $w$ of $W_0$ on its \hwv\ $v_{h,w}$.
If $\kk$ is parametrised by coprime integers $\uu$ and $\vv$, as in \eqref{eq:kuv}, then we let $\irrepsuv$ denote the set of pairs $(h,w)$ such that $\wmodhw$ is a $\wminmoduv$-module.
For nondegenerate levels ($\uu,\vv\ge3$), every irreducible $\wminmoduv$-module is \hw; they were first identified in \cite{FatCon87}.
Here, we use the description of $\irrepsuv$ given in \cite{FehMod21} which is itself an adaptation of the parametrisation used in \cite{BouWSym92}.

For $\ell \in \NN$, let $\pwlat{\ell}$ be the set of triples $t = (t_0,t_1,t_2)$ of nonnegative integers satisfying $t_0+t_1+t_2=\ell$.
Given a nondegenerate level, parametrised by $\uu,\vv\ge3$ as in \eqref{eq:kuv}, consider the set $(\pwlat{\uu-3} \times \pwlat{\vv-3}) / \ZZ_3$, where the $\ZZ_3$-action is simultaneous cyclic permutation:
\begin{equation} \label{eq:Z3action}
	\nabla \colon \brac[\big]{(r_0,r_1,r_2), (s_0,s_1,s_2)} \mapsto \brac[\big]{(r_2,r_0,r_1), (s_2,s_0,s_1)}, \quad r \in \pwlat{\uu-3},\ s \in \pwlat{\vv-3}.
\end{equation}
The classifying set $\irrepsuv$ is, for nondegenerate levels, isomorphic to $(\pwlat{\uu-3} \times \pwlat{\vv-3}) / \ZZ_3$ and an isomorphism is
\begin{subequations} \label{eq:hw(rs)}
	\begin{align}
		h_{[r,s]} &= \frac{1}{3 \uu \vv} \Bigl( \brac[\big]{\vv (r_1+1)-\uu (s_1+1)} \brac[\big]{\vv (r_2+1)-\uu (s_2+1)} \Bigr. \\
		&\mspace{100mu} \Bigl. + \brac[\big]{\vv (r_1+1)-\uu (s_1+1)}^2 + \brac[\big]{\vv (r_2+1)-\uu (s_2+1)}^2 - 3(\vv-\uu)^2 \Bigr), \notag \\
		w_{[r,s]} &= \frac{\brac[\big]{\vv (r_0-r_1) - \uu (s_0-s_1)} \brac[\big]{\vv (r_0-r_2) - \uu (s_0-s_2)} \brac[\big]{\vv (r_1-r_2) - \uu (s_1-s_2)}}{3 (3 \uu \vv)^{3/2}}.
	\end{align}
\end{subequations}
We remark that the vacuum module of $\wminmoduv$ is $\wmod{0}{0}$, corresponding to $[r,s] = [(\uu-3,0,0), (\vv-3,0,0)]$.

\subsection{The \bp\ algebra} \label{sec:bp}

The universal \bp\ algebra $\ubpk$ was introduced in \cite{PolGau90,BerCon91}.
It coincides with the subregular and minimal level-$\kk$ W-algebra associated with $\slthree$ \cite{KacQua03}.
\begin{definition}
	The universal \bp\ algebra $\ubpk$ is the \voa\ strongly generated by four elements $J$, $L$, $G^+$ and $G^-$, subject to the \opes
	\begin{equation} \label{eq:bpopes}
		\begin{gathered}
			\begin{aligned}
				J(z)J(w) &\sim \frac{\kappa \wun}{(z-w)^2}, &
				L(z)G^+(w) &\sim \frac{G^+(w)}{(z-w)^2} + \frac{\pd G^+(w)}{z-w}, \\
				J(z)G^{\pm}(w) &\sim \pm \frac{G^{\pm}(w)}{z-w}, &
				L(z)G^-(w) &\sim \frac{2G^-(w)}{(z-w)^2} + \frac{\pd G^-(w)}{z-w},
			\end{aligned}
			\\
			\begin{aligned}
				L(z)J(w) &\sim -\frac{\kappa \wun}{(z-w)^3} + \frac{J(w)}{(z-w)^2} + \frac{\pd J(w)}{z-w}, \\
				L(z)L(w) &\sim \frac{\ccbp \wun}{2(z-w)^4} + \frac{2L(w)}{(z-w)^2} + \frac{\pd L(w)}{z-w},
			\end{aligned}
			\qquad
			G^{\pm}(z)G^{\pm}(w) \sim 0, \\
			G^+(z)G^-(w) \sim \frac{(\kk+1)(2\kk+3) \wun}{(z-w)^3} + \frac{3(\kk+1)J(w)}{(z-w)^2} + \frac{3\no{J(w)J(w)} + (2\kk+3) \pd J(w) - (\kk+3) L(w)}{z-w}.
		\end{gathered}
	\end{equation}
	Here, $\kk \in \CC \setminus \set{-3}$ is the level, $\kappa$ was defined in \eqref{eq:defkappa} and
	\begin{equation}
		\ccbp = -\frac{4(\kk+1)(2\kk+3)}{\kk+3}.
	\end{equation}
\end{definition}

The universal \bp\ algebra $\ubpk$ is not simple if and only if $\kk$ has the form \eqref{eq:kuv} with $\uu\ge2$ and $\vv\ge1$ \cite{GorSim06}.
In particular, this is the case for all admissible levels.
When $\ubpk$ is not simple, its unique simple quotient, which we shall denote by $\bpminmoduv$, is called a \bp\ minimal model \voa.
Contrary to the case of the $\wminmoduv$, $\bpminmoduv$ is neither rational nor lisse for nondegenerate levels \cite{AdaRea20,FehCla20}.
The same turns out to be true for admissible levels with $\vv=1$ \cite{AdaCla19,AdaBer20}.
However, $\bpminmoduv$ is rational and lisse for admissible levels with $\vv=2$ \cite{AraAss10,AraRat10}, these being exceptional levels in the sense of \cite{AraRat19}.
We remark that unlike the situation for the $\wminmoduv$, there are no isomorphisms between the $\bpminmoduv$ with different parameters.
The trivial case is $\bpminmod{3}{2} \cong \CC$.

We have chosen the conformal vector $L$ of the \bp\ algebra so that the conformal weights of the generating fields are all integral.
The corresponding mode expansions take the form
\begin{equation}
	J(z) = \sum_{n \in \ZZ} J_n z^{-n-1}, \quad
	L(z) = \sum_{n \in \ZZ} L_n z^{-n-2}, \quad
	G^+(z) = \sum_{n \in \ZZ} G^+_n z^{-n-1} \quad \text{and} \quad
	G^-(z) = \sum_{n \in \ZZ} G^-_n z^{-n-2}.
\end{equation}
With this convention, we record the commutation relations of the modes $G^+_m$ and $G^-_n$ for later convenience:
\begin{equation} \label{eq:commGG}
	\begin{split}
		\comm[\big]{G^+_m}{G^-_n} = 3 \no{JJ}_{m+n} - (\kk+3) L_{m+n} &+ \brac[\big]{\kk m - (2\kk+3) (n+1)} J_{m+n} \\
			&+ \tfrac{1}{2} (\kk+1)(2\kk+3) m(m-1) \delta_{m+n,0} \wun.
	\end{split}
\end{equation}

\begin{definition}
	\leavevmode
	\begin{itemize}
		\item The simultaneous eigenspaces of $J_0$ and $L_0$, acting on a $\ubpk$-module, are the module's weight spaces and their nonzero elements are its weight vectors.
		The corresponding weight is the pair $(j,\Delta)$ of $J_0$- and $L_0$-eigenvalues.
		\item A weight $\ubpk$-module is one that is the direct sum of its weight spaces.
		\item A \rhwv\ for $\ubpk$ is a weight vector that is annihilated by every mode with a positive index.
		\item A \hwv\ (conjugate \hwv) for $\ubpk$ is a \rhwv\ that is also annihilated by $G^+_0$ ($G^-_0$).
		\item A (relaxed/conjugate) \hw\ $\ubpk$-module is then one that is generated by a (relaxed/conjugate) \hwv.
	\end{itemize}
\end{definition}

For future work, it is useful to extend these definitions to include $\ubpk$-modules on which $J_0$ acts semisimply but $L_0$ does not.
\begin{definition} \label{def:genwts}
	\leavevmode
	\begin{itemize}
		\item The intersections of the eigenspaces of $J_0$ and the generalised eigenspaces of $L_0$, both acting on a $\ubpk$-module, are the module's generalised weight spaces and their nonzero elements are its generalised weight vectors.
		\item A generalised weight $\ubpk$-module is one that is the direct sum of its generalised weight spaces.
	\end{itemize}
\end{definition}
\noindent Note that for $\ubpk$, an irreducible generalised weight module is always a weight module.

As usual, an irreducible \hw\ $\ubpk$-module $\bpmodjD$ is determined, up to isomorphism, by the weight $(j,\Delta)$ of its \hwv.
One can of course twist the action of $\ubpk$ on $\bpmodjD$ by the conjugation automorphism $\bpconj$ defined by
\begin{equation} \label{eq:defbpconj}
	\begin{aligned}
		\bpconj\brac[\big]{J(z)} &= -J(z) + \kappa z^{-1} \wun, &
		\bpconj\brac[\big]{G^+(z)} &= zG^-(z), \\
		\bpconj\brac[\big]{L(z)} &= L(z) - \pd J(z) - z^{-1} J(z), &
		\bpconj\brac[\big]{G^-(z)} &= -z^{-1}G^+(z),
	\end{aligned}
\end{equation}
as in \eqref{eq:lattsfmod}.
The corresponding functor, also denoted by $\bpconj$, on the category of weight $\ubpk$-modules then yields a bijective correspondence between \hw\ and conjugate \hwms.
As $\bpconj(J_0) = \kappa\wun-J_0$ and $\bpconj(L_0) = L_0$, the weight of the \chwv\ of $\bpconj(\bpmodjD)$ is $(\kappa-j,\Delta)$.

The story is a little different for general irreducible \rhw\ $\ubpk$-modules.
For this case, it will be convenient to introduce some more terminology.
\begin{definition} \label{def:topfullyrelaxed}
	\leavevmode
	\begin{itemize}
		\item The top space of a \rhw\ $\ubpk$-module is the subspace spanned by its vectors of minimal conformal weight.
		\item We shall say that a \rhw\ $\ubpk$-module is \fr, for brevity, if the eigenvalues of $J_0$ on its top space fill out an entire coset in $\CC/\ZZ$.
	\end{itemize}
\end{definition}
\noindent We remark that \hw\ and \chwms\ are relaxed but never \fr.

In the relaxed case, a parametrisation of the irreducibles may be obtained by analysing the Zhu algebra $\zhu{\ubpk}$.
This is known \cite{AraRat10,AdaCla19} to be a central extension of a Smith algebra \cite{SmiCla90}.
Here, we shall think of this Zhu algebra as the zero modes of $\ubpk$ acting on general \rhwvs\ (as in \cite[App.~B]{RidRel15}).
In this framework, $\zhu{\ubpk}$ is generated by $J_0$, $L_0$, $G^+_0$ and $G^-_0$.
As always, $L_0$ is central in this algebra.
\begin{proposition}[\cite{FehCla20}] \label{prop:zhuconsequences}
	\leavevmode
	\begin{enumerate}
		\item\label{it:centraliser} The centraliser in $\zhu{\ubpk}$ of the subalgebra generated by $J_0$ and $L_0$ is $\CC[J_0,L_0,\Omega]$, where the ``cubic Casimir'' $\Omega$ is central and acts on a \rhwv\ $v$ as follows:
		\begin{equation} \label{eq:cubiccasimir}
			\Omega v = \brac*{G^+_0 G^-_0 + G^-_0 G^+_0 + 2J_0^3 - (2\kk+3)J_0^2 + J_0 - 2(\kk+3)J_0 L_0} v.
		\end{equation}
		%\da{ Do we have some Sugawara type of construction for this cubic Casimir?}
		%\dr{I'm not aware of one.  It should be some type of finite hamiltonian reduction of the cubic Casimir of $\slthree$, but I haven't tried to check.}
		\item\label{it:1dimwtspaces} The weight spaces of the top space of an irreducible \rhw\ $\ubpk$-module are $1$-dimensional.
		\item\label{it:trichotomy} An irreducible \rhw\ $\ubpk$-module is either \hw, \chw\ or \fr.
		\item\label{it:isomorphic} An irreducible \fr\ $\ubpk$-module is completely characterised, up to isomorphism, by the equivalence class $[j] \in \CC/\ZZ$ of its $J_0$-eigenvalues, along with the common eigenvalues $\Delta$ of $L_0$ and $\omega$ of $\Omega$ on its top space.
	\end{enumerate}
\end{proposition}
\begin{proof}
	\ref{it:centraliser} is \cite[Lem.~3.20]{FehCla20}.
	It immediately implies \ref{it:1dimwtspaces}, which itself implies \ref{it:trichotomy}.
	We therefore prove \ref{it:isomorphic}.

	It suffices to show \cite{ZhuMod96} that the action of $\zhu{\ubpk}$ on the top space is determined by the weight $(j,\Delta)$ and $\Omega$-eigenvalue $\omega$ of an arbitrarily chosen weight vector $v$ in the top space.
	For this, it is sufficient to show that the actions of $J_0$, $L_0$, $G^+_0$ and $G^-_0$ on a basis of the top space are so determined.
	If $v'$ is a weight vector in the top space, then its $J_0$-eigenvalue is $j+n$, for some $n \in \ZZ$, by irreducibility.
	Irreducibility also means that $v'$ may be obtained from $v$ by acting with some combination of modes.
	Since the \pbw\ theorem holds for the mode algebra of $\ubpk$ \cite[Thm.~4.1]{KacQua03b}, we can actually obtain $v'$ using only zero modes.
	If $n\ge0$, order $G^+_0$ to the left.
	As the weight spaces of the top space are $1$-dimensional, $v'$ can only be obtained if it is a nonzero multiple of $(G^+_0)^n v$.
	Similarly, we see that $v'$ is a nonzero multiple of $(G^-_0)^{-n} v$ for $n\le0$.

	Since our module is \fr, it follows that $\set{v} \cup \set[\big]{(G^+_0)^n v, (G^-_0)^n v \st n>0}$ is a basis of its top space.
	The action of $J_0$ and $L_0$ on these basis vectors is thus fixed by $(j,\Delta)$.
	For $n\ge0$, the action of $G^+_0$ on the $(G^+_0)^n v$ and $G^-_0$ on the $(G^-_0)^n v$ is also clear.
	It therefore remains to check if the action of $G^+_0$ on the $(G^-_0)^n v$ and $G^-_0$ on the $(G^+_0)^n v$, for $n\ge1$, is likewise fixed.
	But, this is clearly the case because $G^+_0 G^-_0$ and $G^-_0 G^+_0$ act on the top space as a polynomial in $J_0$, $L_0$ and $\Omega$, by \eqref{eq:commGG} and \eqref{eq:cubiccasimir}.
\end{proof}

This almost completes the classification of irreducible \rhw\ $\ubpk$-modules --- it only remains to determine which $[j]$, $\Delta$ and $\omega$ actually correspond to irreducible modules.
Rather than delve into the details, we instead make some remarks about the analogous classification for $\bpminmoduv$.

The classification of irreducible \rhw\ $\bpminmoduv$-modules was obtained, for nondegenerate levels, in \cite[Thm.~4.20]{FehCla20} using properties of the minimal \qhr\ functor.
The proof given there is quite subtle, but the result involves the same set $\irrepsuv \cong (\pwlat{\uu-3} \times \pwlat{\vv-3}) / \ZZ_3$ that appears in the classification of irreducible $\wminmoduv$-modules (\cref{sec:zam}).
One of our aims in what follows is to rederive this classification result for $\bpminmoduv$ directly from that for $\wminmoduv$, thereby naturally explaining why this set appears.

To achieve this aim, we shall also need the spectral flow functors $\bpsfsymb^{\ell}$, $\ell \in \ZZ$, on the category of (generalised) weight $\ubpk$-modules.
They are defined in the same way as those introduced on the category of weight $\lvoa$-modules in \cref{sec:latt}, except that $b$ is replaced in \cref{eq:lattsf} by $J$.
For later convenience, we give the action of spectral flow on the modes of the generating fields:
\begin{equation} \label{eq:bpsfJL}
	\bpsfsymb^{\ell}(J_n) = J_n - \kappa \ell \delta_{n,0} \wun, \quad
	\bpsfsymb^{\ell}(G^-_n) = G^-_{n+\ell}, \quad
	\bpsfsymb^{\ell}(G^+_n) = G^+_{n-\ell}, \quad
	\bpsfsymb^{\ell}(L_n) = L_n - \ell J_n + \tfrac{1}{2} \kappa \ell(\ell+1) \delta_{n,0} \wun.
\end{equation}
It is easy to check that the spectral flow and conjugation automorphisms satisfy the dihedral relation
\begin{equation} \label{eq:dihedral}
	\bpsfsymb^{\ell} \bpconj = \bpconj \bpsfsymb^{-\ell}, \quad \ell \in \ZZ.
\end{equation}

Let $v$ be a weight vector of weight $(j,\Delta)$ in some $\bpminmoduv$-module $\Mod{M}$.
The spectral flow action \eqref{eq:lattsfmod} on $\lvoa$-module elements generalises immediately to $\bpminmoduv$-modules (and $\ubpk$-modules) by simply replacing $\lsfsymb$ by $\bpsfsymb$.
Straightforward computation now verifies that $\bpsfsymb^{\ell}(v) \in \bpsfsymb^{\ell}(\Mod{M})$ is a weight vector of weight
\begin{equation} \label{eq:bpsfweight}
	\brac[\big]{j+\kappa\ell, \Delta+\ell j+\tfrac{1}{2}\kappa\ell(\ell-1)}.
\end{equation}
This observation will turn out to be extremely useful in what follows.

\subsection{Inverse \qhr} \label{sec:iqhr}

The idea that one could invert \qhr, in some sense, goes back to \cite{SemInv94}.
However, the crucial observation that this extends to functors between module categories is much more recent \cite{AdaRea17}.
This latter observation was generalised to invert the (partial) reduction of $\ubpk$ to $\uzamk$ and $\bpminmoduv$ to $\wminmoduv$ in \cite{AdaRea20}.
Recall the definition \eqref{eq:theotherbasis} of $a,b \in \lvoa$.
\begin{theorem}[\protect{\cite[Thms.~3.6 and 6.2]{AdaRea20}}] \label{thm:iqhr}
	\leavevmode
	\begin{enumerate}
		\item\label{it:iqhruniv} For $\kk \ne -3$, there is an embedding $\ubpk \hookrightarrow \lvoa \otimes \uzamk$ of universal \voas\ given by
		\begin{equation} \label{eq:iqhr}
			\begin{gathered}
				G^+ \mapsto \ee^c \otimes \wun, \quad J \mapsto b \otimes \wun, \quad L \mapsto t \otimes \wun + \wun \otimes T, \\
				\begin{aligned}
					G^- \mapsto \ee^{-c} \otimes \brac[\Big]{\tfrac{(\kk+3)^{3/2}}{\sqrt{3}} W + \tfrac{1}{2}(\kk+2)(\kk+3) \pd T} + (\kk+3) a_{-1} \ee^{-c} \otimes T& \\
					  - \brac[\big]{a_{-1}^3 + 3 (\kk+2) a_{-2} a_{-1} + 2 (\kk+2)^2 a_{-3}} \ee^{-c} \otimes \wun&.
				\end{aligned}
			\end{gathered}
		\end{equation}
		\item\label{it:iqhrminmod} This descends to an embedding $\bpminmoduv \hookrightarrow \lvoa \otimes \wminmoduv$ of minimal model \voas\ unless $\uu \ge 2$ and $\vv = 1$ or $2$.
		For these $\uu$ and $\vv$, no such embedding of minimal model \voas\ exists.
	\end{enumerate}
\end{theorem}
\noindent Because $J$ is identified with $b$ in \eqref{eq:iqhr}, \cref{thm:iqhr} also identifies the spectral flow maps/functors $\lsfsymb$ and $\bpsfsymb$.
We remark that the embedding of $L$ implies the easily checked identity $\cclvoa + \cczam = \ccbp$.
This identity dictated the choice of conformal structure made in \eqref{eq:defkappa} for $\lvoa$.

\begin{corollary} \label{cor:iqhrmod}
	\leavevmode
	\begin{enumerate}
		\item\label{it:iqhrmoduniv} For $\kk\ne-3$, every $(\lvoa \otimes \uzamk)$-module is a $\ubpk$-module by restriction.
		In particular,
		\begin{equation} \label{eq:iqhrmod}
			\bprmodjhw = \lmod{j} \otimes \wmodhw
		\end{equation}
		is a $\ubpk$-module, for any $[j] \in \CC/\ZZ$ and $h,w \in \CC$.
		\item\label{it:iqhrmodminmod} For $\uu\ge2$ and $\vv\ge3$, every $\brac[\big]{\lvoa \otimes \wminmoduv}$-module is a $\bpminmoduv$-module by restriction.
		In particular, $\bprmodjhw$ is a $\bpminmoduv$-module for all $(h,w) \in \irrepsuv$.
	\end{enumerate}
\end{corollary}
\noindent Recalling that $v_{h,w}$ denotes the \hwv\ of $\wmodhw$, we see that the eigenvalue of $J_0 = b_0 \otimes \wun$ on the \rhwv\ $\ee^{-b+(j+\kappa)c} \otimes v_{h,w} \in \lmod{j} \otimes \wmodhw$ is $j$, explaining the conventions that we chose for the $\lmod{j}$ in \cref{sec:latt}.

Tensoring with a fixed $\lmod{j}$ thus defines a functor from the weight module category of $\uzamk$ to that of $\ubpk$, respectively $\wminmoduv$ and $\bpminmoduv$.
We call these the inverse \qhr\ functors (or just inverse reduction functors for short).
Happily, the modules constructed by these functors turn out to be relevant for classifications.

We recall a useful definition from \cite{AdaRea20}.
\begin{definition}
	A \rhw\ $\ubpk$-module is said to be almost irreducible if it is generated by its top space and all of its nonzero submodules have nonzero intersections with its top space.
\end{definition}
\noindent Of course, an irreducible \rhw\ $\ubpk$-module is almost irreducible.
However, the existence of other almost irreducible $\ubpk$-modules will be crucial for what follows.
\begin{theorem}[\protect{\cite[Cor.~5.11 and Thms.~5.12 and 6.3]{AdaRea20}}] \label{thm:iqhrmod}
	For $\kk\ne-3$, the $\ubpk$-module $\bprmodjhw$:
	\begin{enumerate}
		\item\label{it:almostirred} is indecomposable, almost irreducible and \fr;
		\item\label{it:G+inj} has a bijective action of $G^+_0$;
		\item\label{it:reducible} is, for fixed $h$ and $w$, irreducible for all but at least one, and at most three, $[j] \in \CC/\ZZ$.
	\end{enumerate}
\end{theorem}

Inverse reduction therefore allows us to construct a huge range of irreducible \fr\ $\ubpk$- and $\bpminmoduv$-modules (as well as a few reducible ones) from the irreducible \hwms\ of $\uzamk$ and $\wminmoduv$, respectively.
A natural question is whether every irreducible \frm\ is isomorphic to one that may be so constructed.
When $\kk$ is nondegenerate, the answer is of course yes, by the classification results of \cite{FehCla20}.
However, we seek an answer to this question that is intrinsic to inverse reduction, meaning that it does not rely on comparing with an independent classification theorem.
As further motivation, we want to develop tools to extend the results of \cite{FehCla20} to nonadmissible levels for which the classification is not presently known.

\section{Classifying irreducible weight modules} \label{sec:reclass}

We begin by specifying the module categories of interest.
\begin{definition}
	Let $\wcatk$ and $\wcatuv$ denote the categories of generalised weight $\ubpk$- and $\bpminmoduv$-modules, respectively, with \fdim\ generalised weight spaces (see \cref{def:genwts}).
\end{definition}
\noindent $\wcatuv$ is then a full subcategory of $\wcatk$, where we assume that $\kk$, $\uu$ and $\vv$ are related by \eqref{eq:kuv}.
Much is already known about these categories:
\begin{itemize}
	\item For $\kk \in \CC \setminus \set{-3}$, $\wcatk$ is nonsemisimple with uncountably many irreducible modules (up to isomorphism).
	\item For $\uu,\vv\ge3$ (nondegenerate levels), $\wcatuv$ is also nonsemisimple with uncountably many irreducibles \cite{AdaRea20,FehCla20}.
	\item For $\uu\ge3$, $\wcat_{\uu,2}$ is semisimple with finitely many irreducibles \cite{AraRat10} (in fact, it is a modular tensor category \cite{HuaVer04a}).
	\item For $\uu\ge2$, $\wcat_{\uu,1}$ has uncountably many irreducibles \cite{AdaCla19,AdaBer20}.
	$\wcat_{2,1}$ is semisimple, while the $\wcat_{n,1}$ with $n\ge3$ are not.
\end{itemize}
Our aim here is to use inverse reduction to classify the irreducibles in $\wcatuv$.
This requires the embedding of \cref{thm:iqhr} to exist, so we are limited to studying $\wcatuv$ for nondegenerate levels and nonadmissible levels with $\uu=2$ and $\vv\ge3$.
The classification for these latter levels is currently unknown.

\begin{remark}
	The methods introduced in this \lcnamecref{sec:reclass} may be straightforwardly adapted to prove the analogous classification of irreducible generalised weight modules, with \fdim\ generalised weight spaces, for the simple affine \voa\ $\saff{\kk}{\sltwo}$ with $\kk$ nondegenerate (meaning now that $\kk+2 = \frac{\uu}{\vv}$ with $\uu,\vv\ge2$ coprime).
	We leave the easy details to the reader.
\end{remark}

\subsection{Weight modules for the Heisenberg \va} \label{sec:heis}

We start with a few useful results concerning the Heisenberg \vsa\ $\heis$ of $\ubpk$ generated by $J$.
Abstractly, this vertex algebra admits many choices of conformal vector, each of which yields a nonnegative-integer grading of $\heis$ through the eigenvalues of the associated Virasoro zero mode $L^{\heis}_0$.
Given a choice of grading operator $L^{\heis}_0$, a graded $\heis$-module is then just a module that decomposes as a direct sum of its generalised $L^{\heis}_0$-eigenspaces.

In this \lcnamecref{sec:heis}, any operator $L^{\heis}_0$ satisfying $\comm{L^{\heis}_0}{J_n} = -nJ_n$ will suffice.
For our subsequent applications to $\ubpk$-modules, we will therefore always take the grading operator to be $L_0$ (even though $L \notin \heis$).

The results of this \lcnamecref{sec:heis} are minor modifications of results of Futorny \cite{FutIrr96}; we provide proofs for completeness.
For these, recall that the mode algebra of $\heis$ is (an appropriate completion of) the \uea\ of the affine \km\ algebra $\AKMA{gl}{1}$ (modulo the ideal in which the central element $\wun$ is identified with the \uea's unit).
The latter Lie algebra is spanned by the $J_n$ and $\wun$, with Lie bracket
\begin{equation}
	\comm{J_m}{J_n} = m \delta_{m+n,0} \kappa \wun, \quad
	\comm{J_m}{\wun} = 0, \qquad m,n \in \ZZ.
\end{equation}
The parameter $\kappa$ will be assumed in this \lcnamecref{sec:heis} to be nonzero.
Note that if $v$ is a nonzero vector in an $\heis$-module satisfying $J_n v = 0$ for some $n \ne 0$, then $\kappa \ne 0$ forces $J_{-n} v \ne 0$.

We will also make much use of the operator $A = J_{-1} J_1 \in \uealg{\AKMA{gl}{1}}$.
Its action on a Fock space (\hw\ Verma module) $\fock{j}$, with \hwv\ $v_j$ of $J_0$-eigenvalue $j \in \CC$, picks out the number of $J_{-1}$-modes in each \pbw\ monomial: $A (\cdots J_{-2}^m J_{-1}^n v_j) = n \kappa (\cdots J_{-2}^m J_{-1}^n v_j)$.
Up to the omnipresent factor of $\kappa$, the eigenvalues of $A$ are thus nonnegative integers.

\begin{lemma}[\protect{\cite[Lem.~4.2]{FutIrr96}}] \label{lem:Aeigs}
	Assuming $\kappa\ne0$, let $\Mod{V}$ be a graded $\heis$-module with a nonzero \fdim\ graded subspace $\Mod{V}_{\Delta}$.
	Then, the eigenvalues of $A$ on $\Mod{V}_{\Delta}$ lie in $\kappa \NN$.
\end{lemma}
\begin{proof}
	Since $\Mod{V}_{\Delta}$ is \fdim\ and preserved by the $A$-action, $A$ possesses an eigenvector $v \in \Mod{V}_{\Delta}$.
	Let $\lambda$ denote the associated eigenvalue and assume that $\lambda \notin \kappa \NN$.
	Since $\Mod{V}$ is a module for a \voa, we must have $J_n v = 0$ for $n\gg0$.
	It follows that $J_{-n} v \ne 0$ for $n\gg0$.
	Now consider
	\begin{equation} \label{eq:thetrick}
		J_{-1} J_1^{m+1} J_{-n} v
			= \comm{J_{-1}}{J_1^m} J_1 J_{-n} v + J_1^m A J_{-n} v
			= (\lambda - m \kappa) J_1^m J_{-n} v,
	\end{equation}
	which holds for all $m\ge0$ and $n>1$.
	Since $\lambda \ne 0$, substituting $m=0$ shows that $J_{-1} J_1 J_{-n} v = \lambda J_{-n} v \ne 0$ for $n\gg0$, hence that $J_1 J_{-n} v \ne 0$ for $n\gg0$.
	Substituting successively larger values of $m$, we conclude inductively from $\lambda - m \kappa \ne 0$ that $J_1^m J_{-n} v \ne 0$ for all $m\ge0$ and $n\gg0$.
	In particular, $J_1^n J_{-n} v \in \Mod{V}_{\Delta}$ is nonzero for all $n \gg 0$.
	But,
	\begin{equation}
		A J_1^n J_{-n} v = J_{-1} J_1^{n+1} J_{-n} v = (\lambda - n \kappa) J_1^n J_{-n} v,
	\end{equation}
	so $A$ has infinitely many distinct eigenvalues on $\Mod{V}_{\Delta}$.
	This contradicts $\dim \Mod{V}_{\Delta} < \infty$.
\end{proof}

\begin{lemma}[\protect{\cite[Prop.~4.3]{FutIrr96}}] \label{lem:fock}
	Assuming $\kappa\ne0$, let $\Mod{V}$ be a graded $\heis$-module with a nonzero \fdim\ graded subspace $\Mod{V}_{\Delta}$.
	Then, $\Mod{V}$ has a submodule isomorphic to a Fock space whose \hwv\ has grade $\Delta' \le \Delta$.
\end{lemma}
\begin{proof}
	Again, $A$ has eigenvectors in $\Mod{V}_{\Delta}$ and the eigenvalues all have the form $r \kappa$, with $r \in \NN$, by \cref{lem:Aeigs}.
	Choose an eigenvector $v$ whose eigenvalue $r \kappa$ is such that $r$ is maximal.
	We also assume, without loss of generality, that $v$ is a $J_0$-eigenvector.

	We claim that $J_n v = 0$ for all $n>1$.
	To prove this, suppose that there exists $n>1$ such that $J_n v \ne 0$.
	Then, $J_1 J_{-1}^{m+1} J_n v = (r+m+1) \kappa J_{-1}^m J_n v$ shows inductively that $J_{-1}^m J_n v \ne 0$ for all $m\ge0$, because $r+m+1>0$.
	In particular, $J_{-1}^n J_n v \in \Mod{V}_{\Delta}$ is nonzero, but calculation shows that it is an eigenvector of $A$ with eigenvalue $(r+n) \kappa$.
	Since $n>1$, this contradicts the maximality of $r$ and the claim is proved.

	Consider now the $J_1^m v$ with $m\ge0$.
	If none of these vanish, then $J_n J_1^m v = J_1^m J_n v = 0$ for all $m\ge0$ and $n>1$ implies that $J_{-n} J_1^m v \ne 0$ for all $m\ge0$ and $n>1$.
	But then, $J_{-n} J_1^n v \ne 0$ is an $A$-eigenvector of eigenvalue $(r-n) \kappa$ for all $n>1$, hence this again contradicts $\dim \Mod{V}_{\Delta} < \infty$.
	We conclude that there exists a minimal $m>0$ such that $J_1^m v = 0$.
	It follows that $w = J_1^{m-1} v \ne 0$ is a \hwv\ of grade $\Delta' = \Delta-m+1 \le \Delta$.
	Clearly, it generates the desired Fock space as a submodule of $\Mod{V}$.
\end{proof}

\begin{proposition} \label{prop:lowerbounded}
	Assuming $\kappa\ne0$, let $\Mod{V}$ be a nonzero graded $\heis$-module whose grades all lie in $\Delta + \ZZ$, for some $\Delta \in \CC$.
	Suppose further that all graded subspaces are \fdim.
	Then, the grades of $\Mod{V}$ are bounded below.
\end{proposition}
\begin{proof}
	Choose $\Delta$ so that $\Mod{V}_{\Delta} \ne 0$.
	By \cref{lem:fock}, $\Mod{V}^0 = \Mod{V}$ has a Fock submodule, $\fock{j_0}$ say, whose \hwv\ has grade $\Delta_0 \le \Delta$.
	Since $\fock{j_0}$ is graded with $(\fock{j_0})_{\Delta_0+m} \ne 0$ for all $m \in \NN$, it follows that the quotient module $\Mod{V}^1 = \Mod{V}^0 / \fock{j_0}$ has $\dim \Mod{V}^1_{\Delta} < \dim \Mod{V}^0_{\Delta}$.
	If $\Mod{V}^1_{\Delta} \ne 0$, then \cref{lem:fock} applies and we conclude that $\Mod{V}^1$ has a Fock submodule, $\fock{j_1}$ say, whose \hwv\ has grade $\Delta_1 \le \Delta$.
	Moreover, the quotient $\Mod{V}^2 = \Mod{V}^1 / \fock{j_1}$ has $\dim \Mod{V}^2_{\Delta} < \dim \Mod{V}^1_{\Delta}$.
	Continuing, we obtain a sequence of quotient $\heis$-modules $\Mod{V}^m = \Mod{V}^{m-1} / \fock{j_{m-1}}$ and Fock submodules $\fock{j_m} \subseteq \Mod{V}^m$ whose \hwvs\ have grades $\Delta_m \le \Delta$.
	Because the dimension of $\Mod{V}^m_{\Delta}$ is strictly decreasing, there exists $n$ such that $\Mod{V}^n_{\Delta} = 0$.

	We claim that in fact $\Mod{V}^n_{\Delta'} = 0$ for all $\Delta' \le \Delta$.
	Suppose not, so that there exists $\Delta' < \Delta$ with $\Mod{V}^n_{\Delta'} \ne 0$.
	Then, $\Mod{V}^n_{\Delta'}$ is \fdim, because $\Mod{V}_{\Delta'}$ is, hence \cref{lem:fock} applies and $\Mod{V}^n$ has a Fock submodule $\fock{j_n}$ whose \hwv\ has grade $\Delta_n \le \Delta' < \Delta$.
	But, this is impossible because $(\fock{j_n})_{\Delta} \ne 0$ while $\Mod{V}^n_{\Delta} = 0$.
	This proves that $\Mod{V}^n_{\Delta'} = 0$ for all $\Delta' \le \Delta$, hence that $\Mod{V}$ has a minimal grade (the minimum of the $\Delta_m$, $m=0,1,\dots,n-1$).
\end{proof}

It is perhaps useful to finish with an example that illustrates the need for a \fdim ity hypothesis in \cref{prop:lowerbounded}.
Consider the triangular decomposition of $\AKMA{gl}{1}$ into the following three Lie subalgebras:
\begin{equation}
	\alg{n}_- = \spn \set{J_{-n}, J_1 \st n\ge2}, \quad
	\alg{h} = \spn \set{J_0, \wun}, \quad
	\alg{n}_+ = \spn \set{J_n, J_{-1} \st n\ge2}.
\end{equation}
Setting $\alg{b} = \alg{h} \oplus \alg{n}_+$, we consider the $\alg{b}$-module $\CC v$ defined by $J_0 v = \alg{n}_+ v = 0$ and $\wun v = v$.
The associated Verma module $\uealg{\AKMA{gl}{1}} \otimes_{\uealg{\alg{b}}} \CC v$ has a \pbw\ basis consisting of monomials of the form $\cdots J_{-3}^{\ell} J_{-2}^m J_1^n v$.
This $\AKMA{gl}{1}$-module is clearly graded with grades that differ by integers.
However, the grades are neither bounded above nor below.
More interestingly, it is a smooth $\AKMA{gl}{1}$-module (in the sense of \cite{FreVer01}), hence it is an $\heis$-module.
This is nevertheless consistent with \cref{prop:lowerbounded} because its graded subspaces are all \infdim.

\subsection{Extremal weights} \label{sec:wtcat}

We now return to our study of the categories $\wcatk$ and $\wcatuv$ of generalised weight $\ubpk$- and $\bpminmoduv$-modules, respectively, with \fdim\ generalised weight spaces.
\begin{definition} \label{def:extwt}
	\leavevmode
	\begin{itemize}
		\item An extremal weight of a $\ubpk$-module $\Mod{M}$ is a weight $(j,\Delta)$ whose $L_0$-eigenvalue $\Delta$ is minimal among those of all weights sharing the same $J_0$-eigenvalue $j$.
		\item $\Mod{M}$ is said to admit extremal weights if there is an extremal weight for each eigenvalue of $J_0$ on $\Mod{M}$.
	\end{itemize}
\end{definition}
\noindent Consider any $\ubpk$-module in $\wcatk$ whose $L_0$-eigenvalues all lie in $\Delta + \ZZ$, for some $\Delta \in \CC$.
For example, any indecomposable module in $\wcatk$ has this property.
Then, its $J_0$-eigenspaces are $\heis$-modules to which \cref{prop:lowerbounded} applies, as long as $\kappa\ne0$.
Assuming this, it follows that each $J_0$-eigenspace has an extremal weight, hence the $\ubpk$-module admits extremal weights.

A slightly more general consequence of \cref{prop:lowerbounded} is then as follows.
\begin{proposition} \label{prop:extwtsexist}
	\leavevmode
	\begin{enumerate}
		\item For $\kk\ne-3, -\frac{3}{2}$, every finitely generated module in $\wcatk$ admits extremal weights.
		\item For coprime integers $\uu\ge2$ and $\vv\ge1$, every finitely generated module in $\wcatuv$ admits extremal weights.
	\end{enumerate}
\end{proposition}
\begin{proof}
	These follow immediately as above, except when $\uu=3$ and $\vv=2$, hence $\kk=-\frac{3}{2}$ and $\kappa=0$.
	In this case, the \bp\ minimal model \voa\ is trivial: $\bpminmod{3}{2} \cong \CC$.
	The finitely generated $\bpminmod{3}{2}$-modules are thus finite direct sums of the $1$-dimensional module and they clearly admit extremal weights.
	In fact, they have a unique extremal weight: $(0,0)$.
\end{proof}

\begin{remark} \label{rem:secondcase}
	A second exceptional case occurs when $\kk=-1$, equivalently $\uu=2$ and $\vv=1$, because the \bp\ minimal model then reduces to the Heisenberg \va\ $\heis$ \cite{AdaCon16}.
	In this case, the Fock modules are the irreducible modules in $\wcat_{2,1}$ and they also have a unique extremal weight: $\brac[\big]{j,\frac{1}{2}j(3j-1)}$.
\end{remark}

\begin{lemma} \label{lem:extwtsexist}
	For $\kk\ne-3, -1, -\frac{3}{2}$, the extremal weights of any irreducible module in $\wcatk$ have the form $(j,\Delta_j)$, where $j$ runs over a complete equivalence class in $\CC/\ZZ$.
\end{lemma}
\begin{proof}
	Obviously, the set of $J_0$-eigenvalues on any irreducible weight $\ubpk$-module must be contained in a single equivalence class in $\CC/\ZZ$.
	Suppose that the set of $J_0$-eigenvalues of the extremal weights of an irreducible module $\Mod{M}$ in $\wcatk$ has a ``gap'' for which $j$ belongs to this set but $j-1$ does not.
	(The other possibility, that $j+1$ does not belong, follows from this one by applying conjugation.)

	Then, there exists a weight vector $v \in \Mod{M}$ of $J_0$-eigenvalue $j$ and we must have $G^-_m v = 0$ for all $m \in \ZZ$.
	As $\Mod{M}$ is a module over a \voa, we also have $G^+_n v = 0$ for all $n\gg0$.
	This implies that $\comm{G^+_n}{G^-_{-n}} v = 0$ for all $n\gg0$.
	In particular, \eqref{eq:commGG} gives
	\begin{equation}
			0 = \brac[\big]{\comm{G^+_{n+1}}{G^-_{-n-1}} - \comm{G^+_n}{G^-_{-n}}} v
			= 3(\kk+1) J_0 v + (\kk+1) (2\kk+3) n v
			= (\kk+1) \brac[\big]{3j + (2\kk+3) n} v,
	\end{equation}
	for all $n\gg0$.
	This is only possible if either $\kk=-1$ or both $\kk=-\frac{3}{2}$ and $j=0$ hold.
	Otherwise, the set of $J_0$-eigenvalues cannot have such a gap.
\end{proof}

\begin{lemma} \label{lem:neighbours}
	If $(j-1,\Delta-m)$, $(j,\Delta)$ and $(j+1,\Delta+n)$ are extremal weights of an irreducible module $\Mod{M} \in \wcatk$, then $m \le n$.
\end{lemma}
\begin{proof}
	Since the \pbw\ theorem holds for the mode algebra of $\ubpk$ \cite[Thm.~4.1]{KacQua03b}, we may choose an ordering so that monomials have the $G^-_r$, with $r>n$, as the rightmost modes and the $G^-_r$, with $r \le n$, as the leftmost.
	With this ordering, every monomial that maps the extremal weight $(j+1,\Delta+n)$ into the extremal weight $(j,\Delta)$ has $G^-_n$ as its leftmost mode.
	Similarly, every monomial mapping the extremal weight $(j+1,\Delta+n)$ into the extremal weight $(j-1,\Delta-m)$ has $G^-_{n_1} G^-_{n_2}$ as its two leftmost modes, where $n_1, n_2 \le n$ and $n_1 + n_2 = m+n$.
	If $m>n$, then there are no such monomials.
	However, this contradicts the assumption that $\Mod{M}$ is irreducible.
\end{proof}
\noindent Note that \cref{lem:extwtsexist} establishes that the hypothesised extremal weights in \cref{lem:neighbours} always exist as long as $\kk\ne-3, -1, -\frac{3}{2}$.
\begin{theorem} \label{prop:irred=sf+relax}
	\leavevmode
	\begin{enumerate}
		\item For $\kk\ne-3, -1, -\frac{3}{2}$, every irreducible module in $\wcatk$ is the spectral flow of a \rhwm.
		\item For coprime integers $\uu\ge2$ and $\vv\ge1$, every irreducible module in $\wcatuv$ is the spectral flow of a \rhwm.
	\end{enumerate}
\end{theorem}
\begin{proof}
	We prove the statement for $\wcatk$, noting that the statement for $\wcatuv$ follows because we have already noted that each irreducible module in $\wcat_{2,1}$ and $\wcat_{3,2}$ is \hw\ (see \cref{prop:extwtsexist,rem:secondcase}).

	So take $\kk\ne-3, -1, -\frac{3}{2}$ and fix an irreducible module $\Mod{M} \in \wcatk$.
	Let $(j,\Delta_j)$ denote its extremal weights, where $j$ runs over an equivalence class in $\CC/\ZZ$ (\cref{lem:extwtsexist}).
	Defining $\delta_j(\Mod{M}) = \Delta_{j+1} - \Delta_j$, it follows from \cref{lem:neighbours} that $\delta_j(\Mod{M})$ is weakly increasing with $j$.
	The limiting values $\delta_{\infty}(\Mod{M})$ and $\delta_{-\infty}(\Mod{M})$ are then defined, though they may be $\infty$ and $-\infty$, respectively.

	Suppose first that $\delta_{\infty}(\Mod{M}) \ge 0$ and $\delta_{-\infty}(\Mod{M}) \le 0$.
	Then, it follows that $\Delta_j$ must take a minimal value.
	Choose any $j$ such that $\Delta_j$ achieves this global minimum.
	Then, the corresponding weight vectors are \rhwvs.
	As $\Mod{M}$ is irreducible, it is thus a \rhwm.

	Suppose next that $\delta_{-\infty}(\Mod{M}) > 0$, hence that $\delta_{\infty}(\Mod{M}) > 0$ too.
	Then, $\Delta_j$ has no minima and $\Mod{M}$ is not \rhw.
	However, \eqref{eq:bpsfweight} shows that spectral flow maps extremal weights to extremal weights.
	It also shows that applying the functor $\bpsfsymb^{\ell}$ increases $\delta_j$ by $\ell$:
	\begin{equation}
		\delta_{j+(2\kk+3)\ell/3}\brac[\big]{\bpsfsymb^{\ell}(\Mod{M})} = \delta_j(\Mod{M}) + \ell.
	\end{equation}
	Taking $\ell = -\delta_{-\infty}(\Mod{M})$ and $j \to \pm\infty$ then gives $\delta_{\infty}\brac[\big]{\bpsfsymb^{\ell}(\Mod{M})} = \delta_{\infty}(\Mod{M}) - \delta_{-\infty}(\Mod{M}) \ge 0$ and $\delta_{-\infty}\brac[\big]{\bpsfsymb^{\ell}(\Mod{M})} = 0$.
	We therefore conclude that $\bpsfsymb^{-\delta_{-\infty}(\Mod{M})}(\Mod{M})$ is a \rhwm, by the previous part.

	The only remaining possibility is that $\delta_{\infty}(\Mod{M}) < 0$, which requires that $\delta_{-\infty}(\Mod{M}) < 0$ as well.
	In this case, a similar argument shows that $\bpsfsymb^{-\delta_{\infty}(\Mod{M})}(\Mod{M})$ is \rhw.
\end{proof}

\subsection{Completeness for irreducible \frms} \label{sec:complete}

The previous \lcnamecref{sec:wtcat} reduced the classification of irreducible modules in $\wcatk$ and $\wcatuv$ to the classification of \rhwms.
In this \lcnamecref{sec:complete}, we shall establish that the inverse reduction functors, when defined by \eqref{eq:iqhrmod}, construct a complete set of irreducible \fr\ $\bpminmoduv$-modules.

Recall that $\bprmodjhw = \lmod{j} \otimes \wmodhw$ is a $\ubpk$-module, for all $\kk\ne-3$, by \cref{cor:iqhrmod}\ref{it:iqhrmoduniv}.
\begin{proposition} \label{prop:completeuniv}
	For $\kk\ne-3$, every irreducible \fr\ $\ubpk$-module $\Mod{M}$ is isomorphic to $\bprmodjhw$, for some $[j] \in \CC/\ZZ$ and $h,w \in \CC$.
\end{proposition}
\begin{proof}
	As $\Mod{M}$ is irreducible and \rhw, it is determined up to isomorphism by the eigenvalues of $J_0$, $L_0$ and $\Omega$ on some weight vector $v$ in its top space, by \cref{prop:zhuconsequences}\ref{it:isomorphic}.
	Let $j'$, $\Delta$ and $\omega$ denote these eigenvalues, respectively.
	Then, we need only match them with those of some \rhwv\ $\ee^{-b+(j+n+\kappa)c} \otimes v_{h,w}$, $n \in \ZZ$, in $\bprmodjhw$.
	Here, $v_{h,w}$ is the \hwv\ of $\wmodhw$.

	As noted after \cref{cor:iqhrmod}, the $J_0$-eigenvalue is $j+n$.
	This means that we must choose $j'=j+n$, for some $n \in \ZZ$, hence $[j']=[j]$ in $\CC/\ZZ$.
	A similar computation with $L_0$ instead of $J_0$ leads to $\Delta = h+\kappa$.
	The computation for $\Omega$ is complicated by the form of $G^-_0$ in \eqref{eq:iqhr}.
	However, it is enough to note that
	\begin{subequations} \label{eq:actionofG-}
		\begin{equation}
			G^-_0 (\ee^{-b+(j+n+\kappa)c} \otimes v_{h,w}) = \brac[\big]{\alpha_{\kk} w + P_{\kk}(j+n,h)} \ee^{-b+(j+n-1+\kappa)c} \otimes v_{h,w},
		\end{equation}
		where $\alpha_{\kk} = \frac{(\kk+3)^{3/2}}{\sqrt{3}}$ and $P_{\kk}$ is the polynomial
		\begin{equation}
			 P_{\kk}(j,h) = -(\kk+2)(\kk+3) h + \brac[\big]{(\kk+3) h - 2 (\kk+2)^2} (j+\kappa) + 3(\kk+2) (j+\kappa)^2 - (j+\kappa)^3.
		\end{equation}
	\end{subequations}
	In fact, the precise form of this polynomial is unimportant here.
	All we need is that the consequent identification for the $\Omega$-eigenvalue has the form $\omega = 2 \alpha_{\kk} w + Q_{\kk}(j,h)$, where $Q_{\kk}$ is a (different) polynomial in $j$ and $h$, by \eqref{eq:cubiccasimir}.
	(Because $\Omega$ is central in the Zhu algebra of $\ubpk$ and $\Mod{M}$ is irreducible, this polynomial is in fact independent of $j$.)

	We conclude that any choice of $[j']$, $\Delta$ and $\omega$ corresponds to some (unique) choice of $[j]$, $h$ and $w$.
\end{proof}
\noindent We mention that while the precise form of the polynomial $P_{\kk}$ was not important for the proof of \cref{prop:completeuniv}, it will be important in some of the finer classification analyses in \cref{sec:examples}.
\begin{remark} \label{rem:localisation}
	There is an alternative means to prove \cref{prop:completeuniv} that may be more useful when generalising to higher rank W-algebras.
	First, prove the corresponding statement for \hw\ (or \chw) modules.
	This is somewhat easier because the eigenvalues that one is required to match will not include those of any ``higher Casimir'' operators.
	Then, extend the proof to \frms\ using the analogue of Mathieu's twisted localisation functors \cite{MatCla00} for the W-algebra's Zhu algebra, as in \cite{KawRel19,KawRel20}.
\end{remark}

This establishes the desired completeness result for the universal \bp\ \voas.
We next turn to its analogue for irreducible \fr\ $\bpminmoduv$-modules.
This means restricting to $\uu\ge2$ and $\vv\ge3$, by \cref{thm:iqhr}\ref{it:iqhrminmod}.
We mention again that the irreducible \rhw\ $\bpminmoduv$-modules with $\uu\ge2$ and $\vv=1$ or $2$ were already shown to be \hw\ in \cite{AdaCla19}.

We start with a technical \lcnamecref{lem:stringfns} about the embedding \eqref{eq:iqhr} of universal \voas\ given in \cref{thm:iqhr}\ref{it:iqhruniv}.
Our proof involves characters and string functions, although it is also easy to give an equivalent combinatorial proof using \pbw\ bases.
\begin{lemma} \label{lem:stringfns}
	For every $v \in \uzamk$, we have $\ee^{nc} \otimes v \in \ubpk$ for all $n\gg0$.
\end{lemma}
\begin{proof}
	We may assume, without loss of generality, that $v$ is a weight vector of weight $h$ (say).
	Then, the statement of the \lcnamecref{lem:stringfns} will follow if we can show that the dimensions of the weight spaces of $\lvoa \otimes \uzamk$ and $\ubpk$, with weight $(n,h+n)$, match for $n\gg0$.
	For this, it suffices to show that the string functions $s^{\bpsymb}_n(q)$ of $\ubpk$ converge to the string functions $s_n(q)$ of $\lvoa \otimes \uzamk$ as $n \to \infty$.
	(We define these string functions below and refer to \cite[App~A]{KawRel18} for further details.)

	Define the characters of $\lvoa$ and $\uzamk$ as follows:
	\begin{equation}
		\begin{aligned}
			\fch{\lvoa}{z;q} &= \traceover{\lvoa} z^{b_0} q^{t_0-\cclvoa/24}
			= \frac{\sum_{n\in\ZZ} z^n q^{n-\cclvoa/24}}{\prod_{i=1}^{\infty} (1-q^i)^2}, \\
			\fch{\uzamk}{q} &= \traceover{\uzamk} q^{T_0-\cczam/24}
			= \frac{q^{-\cczam/24}}{\prod_{i=1}^{\infty} (1-q^{i+1}) (1-q^{i+2})}.
		\end{aligned}
	\end{equation}
	The string function $s_n(q)$ of $\lvoa \otimes \uzamk$ is then the coefficient of $z^n$ in its character:
	\begin{equation} \label{eq:stringfn1}
		s_n(q) = \frac{q^{n-\ccbp/24}}{\prod_{i=1}^{\infty} (1-q^i)^2 (1-q^{i+1})(1-q^{i+2})}.
	\end{equation}
	We note that $q^{-n} s_n(q)$ is independent of $n$.
	For this reason, we shall actually prove that the string function of $\ubpk$, normalised by a factor of $q^{-n}$, converges as $n \to \infty$ to $q^{-n} s_n(q)$.

	To do this, we employ the method of \cite[App.~A]{KawRel18}.
	First, note that $J \mapsto b \otimes \wun$ in \eqref{eq:iqhr} implies that the appropriate definition of character for $\ubpk$ is
	\begin{equation} \label{eq:bpch1}
		\fch{\ubpk}{z;q} = \traceover{\ubpk} z^{J_0} q^{L_0-\ccbp/24}
		= \frac{q^{-\ccbp/24}}{\prod_{i=1}^{\infty} (1-q^i) (1-zq^i) (1-q^{i+1})(1-z^{-1}q^{i+1})}.
	\end{equation}
	(This is in fact the standard definition, explaining why we defined $\ch{\lvoa}$ as we did above.)
	Next, note that as $\ubpk$ has \fdim\ $L_0$-eigenspaces, its character (as a power series) must converge when $z=1$.
	Looking at the poles in \eqref{eq:bpch1}, we conclude that the \rhs\ will give the correct power series when expanded in the region $\abs{q}<1$ and $\abs{q}^2 < \abs{z} < \abs{q}^{-1}$.
	In particular, we may assume that $\abs{zq} < 1$.

	This motivates writing \eqref{eq:bpch1} in the form
	\begin{equation}
		\begin{split}
			\fch{\ubpk}{z;q}
			&= \frac{q^{-\ccbp/24}}{1-zq} \frac{1}{\prod_{i=1}^{\infty} (1-q^i) (1-q^{i+1}) (1-zq^{i+1})(1-z^{-1}q^{i+1})} \\
			&= q^{-\ccbp/24} \sum_{\ell=0}^{\infty} z^{\ell} q^{\ell} \sum_{m=0}^{\infty} p_m(z) q^m,
		\end{split}
	\end{equation}
	where the $p_m$ are Laurent polynomials.
	The string function is then obtained as a residue about $0$:
	\begin{equation}
		s^{\bpsymb}_n(q) = \oint_0 \fch{\ubpk}{z;q} z^{-n-1} \, \frac{\dd z}{2\pi\ii}
		= \sum_{m=0}^{\infty} \sum_{\ell=0}^{\infty} \oint_0 p_m(z) z^{\ell-n-1} \, \frac{\dd z}{2\pi\ii} \, q^{\ell+m-\ccbp/24}.
	\end{equation}
	For each $m$, we see that the $\ell$-sum may be extended to include the negative integers, provided that $n$ is larger than the maximal power of $z$ appearing in $p_m(z)$.
	In particular, this extension is justified in the $n \to \infty$ limit, giving
	\begin{equation}
		\begin{split}
			\lim_{n\to\infty} q^{-n} s^{\bpsymb}_n(q)
			&= \lim_{n\to\infty} q^{-n-\ccbp/24} \oint_0 \sum_{\ell=-\infty}^{\infty} z^{\ell} q^{\ell} \sum_{m=0}^{\infty} p_m(z) q^m \, z^{-n-1} \, \frac{\dd z}{2\pi\ii} \\
			&= \lim_{n\to\infty} q^{-\ccbp/24} \oint_0 \frac{\delta(zq) \, (zq)^{-n} z^{-1}}{\prod_{i=1}^{\infty} (1-q^i) (1-q^{i+1}) (1-zq^{i+1})(1-z^{-1}q^{i+1})} \, \frac{\dd z}{2\pi\ii} \\
			&= \frac{q^{-\ccbp/24}}{\prod_{i=1}^{\infty} (1-q^i)^2 (1-q^{i+1})(1-q^{i+2})} \oint_0 \delta(zq) \, z^{-1} \, \frac{\dd z}{2\pi\ii}
			= q^{-n} s_n(q),
		\end{split}
	\end{equation}
	as desired.
	Here, $\delta(x) = \sum_{\ell \in \ZZ} x^{\ell}$ denotes the delta function of formal power series.
\end{proof}

Recall that for coprime integers $\uu\ge2$ and $\vv\ge3$, $\irrepsuv$ is the set of pairs $(h,w) \in \CC^2$ such that the irreducible \hw\ $\uzamk$-module $\wmodhw$ is a $\wminmoduv$-module.
\cref{cor:iqhrmod}\ref{it:iqhrmodminmod} then guarantees that the $\bprmodjhw$, with $(h,w) \in \irrepsuv$, are $\bpminmoduv$-modules.
We now prove a converse.
\begin{theorem} \label{prop:completeminmod}
	For coprime integers $\uu\ge2$ and $\vv\ge3$, every irreducible \fr\ $\bpminmoduv$-module is isomorphic to $\bprmodjhw$, for some $[j] \in \CC/\ZZ$ and $(h,w) \in \irrepsuv$.
\end{theorem}
\begin{proof}
	By \cref{prop:completeuniv}, every irreducible \fr\ $\bpminmoduv$-module is isomorphic to $\bprmodjhw$, for some $[j] \in \CC/\ZZ$ and $(h,w) \in \CC^2$.
	(Obviously, this means that this $\bprmodjhw$ is also irreducible and \fr.)
	Our task is thus to prove that in fact $(h,w) \in \irrepsuv$.

	Suppose that this is not the case, so that $\wmodhw$ is not a $\wminmoduv$-module.
	We will show that this implies that $\bprmodjhw = \lmod{j} \otimes \wmodhw$ is not a $\bpminmoduv$-module.
	To see this, let $\zammaxpropk$ be the (unique) maximal ideal of $\uzamk$.
	Then, $\zammaxpropk \cdot \wmodhw \ne 0$.
	In fact, as $\wmodhw$ is generated by its \hwv\ $v_{h,w}$, we must have $\zammaxpropk v_{h,w} \ne 0$.
	In other words, there exists $\chi \in \zammaxpropk$ and $m \in \ZZ$ such that $\chi_m v_{h,w} \ne 0$.
	We shall choose $m$ to be maximal with this property.

	By \cref{lem:stringfns}, there exists $n \in \ZZ$ such that $\ee^{nc} \otimes \chi \in \ubpk$.
	We claim that for this $n$, $\ee^{nc} \otimes \chi$ is necessarily in the (unique) maximal ideal $\bpmaxpropk$ of $\ubpk$.
	For if this were not the case, then we could act with $\ubpk$-modes to bring $\ee^{nc} \otimes \chi$ to the vacuum vector $\wun_{\bpsymb} = \wun_{\lvoa} \otimes \wun_{\zamsymb}$.
	However, this is impossible because $\chi \in \zammaxpropk$.
	On the other hand, the maximality of $m$ gives
	\begin{equation} \label{eq:itain'tzero}
		(\ee^{nc} \otimes \chi)_m (\ee^{-b+(j+\kappa)c} \otimes v_{h,w})
		= \ee^{nc}_0 \ee^{-b+(j+\kappa)c} \otimes \chi_m v_{h,w}
		= \ee^{-b+(j+n+\kappa)c} \otimes \chi_m v_{h,w} \ne 0.
	\end{equation}
	This shows that there is an element of $\bpmaxpropk$ acting nontrivially on an element of $\bprmodjhw$, proving that $\bprmodjhw$ is not a $\bpminmoduv$-module, as required.
\end{proof}

We recall from \cref{sec:zam} that for nondegenerate levels, $\irrepsuv$ is known to be isomorphic to $(\pwlat{\uu-3} \times \pwlat{\vv-3}) / \ZZ_3$, see \cref{eq:Z3action,eq:hw(rs)}.
Consequently, \cref{prop:completeminmod} recovers the \fr\ part of the irreducible classification in $\wcatuv$ for these levels, as was first established in \cite{FehCla20} using different methods.
Crucially, the inverse reduction arguments given here explain why the set $(\pwlat{\uu-3} \times \pwlat{\vv-3}) / \ZZ_3$ appears in this result.

\subsection{Completeness for irreducible \hwms} \label{sec:completehw}

Having classified the irreducible \fr\ $\bpminmoduv$-modules, for $\uu\ge2$ and $\vv\ge3$, we turn to the remaining irreducible \rhwms.
As noted in \cref{prop:zhuconsequences}\ref{it:trichotomy}, these are either \hw\ or \chw.
We shall start by classifying the \hw\ $\bpminmoduv$-modules with an \infdim\ top space.

\begin{proposition} \label{prop:infdimhwm}
	Let $\kk\ne-3$ (coprime integers $\uu\ge2$ and $\vv\ge3$).
	Then:
	\begin{enumerate}
		\item\label{it:infdimchw} Every irreducible \chw\ $\ubpk$-module ($\bpminmoduv$-module) $\bpcmodjD$ with an \infdim\ top space is isomorphic to a submodule of $\bprmodjhw$, for some unique $(h,w) \in \CC^2$ ($(h,w) \in \irrepsuv$).
		\item\label{it:infdimhw} Every irreducible \hw\ $\ubpk$-module ($\bpminmoduv$-module) $\bpmodjD$ with an \infdim\ top space is isomorphic to a quotient of $\bprmodjhw$, for some unique $(h,w) \in \CC^2$ ($(h,w) \in \irrepsuv$).
	\end{enumerate}
\end{proposition}
\begin{proof}
	As in the proof of \cref{prop:completeuniv}, we search for a \chwv\ of weight $(j+1,\Delta)$ among the \rhwvs\ $\ee^{-b+(j'+n+\kappa)c} \otimes v_{h,w}$, $n \in \ZZ$, of $\bprmod{j'}{h}{w}$.
	Calculation with \eqref{eq:iqhr} shows that the weights match if we take $j'+n=j+1$, hence $[j']=[j]$, and $h=\Delta-\kappa$, while being a \chwv\ fixes $w$ uniquely in terms of $j$ and $\Delta$, by \eqref{eq:actionofG-}.
	This vector then generates a \chw\ submodule $\Mod{C}$ of $\bprmodjhw$.
	Evidently, $\bpcmod{j+1}{\Delta}$ is a quotient of $\Mod{C}$.
	However, every nonzero submodule of $\bprmodjhw$ has a nonzero intersection with its top space (\cref{thm:iqhrmod}\ref{it:almostirred}), hence the same is true for its submodule $\Mod{C}$.
	If $\bpcmod{j+1}{\Delta} \ncong \Mod{C}$, then $\Mod{C}$ has a submodule whose intersection with its top space is nonzero.
	However, this is impossible as the top spaces of both $\bpcmod{j+1}{\Delta}$ and $\Mod{C}$ have the same weights, $(j+n+1,\Delta)$ for all $n\ge0$, appearing with the same multiplicities, while $\bpcmod{j+1}{\Delta}$ is irreducible.
	We conclude that $\bpcmod{j+1}{\Delta} \ncong \Mod{C}$, proving \ref{it:infdimchw} for $\ubpk$.

	For \ref{it:infdimhw}, note that the top space of the quotient $\Mod{Q} = \bprmodjhw / \Mod{C}$ has weights $(j-n,\Delta)$, for all $n\ge0$.
	Consequently, $\Mod{Q}$ has a \hwv\ of weight $(j,\Delta)$.
	Let $\Mod{H}$ be the \hw\ submodule of $\Mod{Q}$ generated by this \hwv.
	As $\bpmodjD$ is irreducible, it is a quotient of $\Mod{H}$.
	Because its top space is \infdim, its top space's weights precisely match those of $\Mod{Q}$, hence so do those of $\Mod{H}$.
	By \cref{prop:zhuconsequences}\ref{it:1dimwtspaces}, the top spaces of $\Mod{H}$ and $\Mod{Q}$ therefore coincide.
	But, $\bprmodjhw$ is generated by its top space, by \cref{thm:iqhrmod}\ref{it:almostirred}, hence the same is true for $\Mod{Q}$.
	It follows that $\Mod{H} = \Mod{Q}$, hence that $\bpmodjD$ is a quotient of $\Mod{Q}$ and, thus, also of $\bprmodjhw$.
	This completes the proof for $\ubpk$-modules.

	To finish, we only need to show that $\bpcmod{j+1}{\Delta}$ or $\bpmodjD$ being a $\bpminmoduv$-module implies that $\bprmod{j'}{h}{w}$ is too.
	This is essentially \cite[Prop.~4.22]{FehCla20} (see also \cite[Thm.~5.3]{KawRel19}).
	We sketch the proof for $\bpmodjD$ for completeness, leaving that for $\bpcmod{j+1}{\Delta}$ as an exercise.

	Recall that $\bprmodjhw$ is generated by its top space.
	In fact, it is generated by any of its top space weight vectors as long as the $J_0$-eigenvalue is at most $j$.
	This follows as $G^+_0$ acts bijectively on the weight spaces of the top space while $G^-_0$ acts bijectively on those with $J_0$-eigenvalue at most $j$ (because the quotient $\bpmodjD$ is irreducible).

	Since $\zhu{\ubpk}$ is noetherian \cite{SmiCla90}, its maximal ideal is generated by a finite number of $J_0$-eigenvectors $A^{(i)}_0$, say.
	Choose a positive $n$ greater than the $J_0$-eigenvalues of the $A^{(i)}_0$ and pick a weight vector $v$ of weight $(j-n,\Delta)$ in the top space of $\bprmodjhw$.
	Then, $v$ generates $\bprmodjhw$.
	Its image $\overline{v}$ in $\bpmodjD$ is annihilated by the $A^{(i)}_0$ because $\bpmodjD$ is a $\bpminmoduv$-module.
	It follows that $A^{(i)}_0 v$ must lie in a weight space of the maximal submodule of $\bprmodjhw$, the quotient by which is $\bpmodjD$.
	However, the $J_0$-eigenvalue of $A^{(i)}_0 v$ is not greater than $j$, for all $i$, by construction.
	The weight space of the maximal submodule is therefore $0$, so $A^{(i)}_0 v = 0$ for all $i$.
	We conclude that the maximal ideal of $\zhu{\ubpk}$ annihilates a vector $v$ in the top space of $\bprmodjhw$ that generates the entire module.
	This proves that $\bprmodjhw$ is a $\bpminmoduv$-module, as desired.
\end{proof}

This implies that we can obtain a complete set of irreducible \hw\ $\bpminmoduv$-modules, with \infdim\ top spaces, by identifying the irreducible quotient of each reducible $\bprmodjhw$.
A complete set of irreducible \chwms, again with \infdim\ top spaces, is then obtained by applying the conjugation functor.
It only remains to study the irreducible \hw\ $\bpminmoduv$-modules, with \fdim\ top spaces.

\begin{proposition} \label{prop:fdimhwm}
	For $\kk \notin \set{-3} \cup \frac{1}{2} \ZZ_{\ge-3}$, the spectral flow orbit $\bpsforbit{\Mod{H}} = \set[\big]{\bpsfsymb^{\ell}(\Mod{H}) \st \ell \in \ZZ}$ of any irreducible \hw\ $\ubpk$-module $\Mod{H}$ contains:
	\begin{enumerate}
		\item\label{it:sforbhwinf} a unique \hwm\ whose top space is \infdim;
		\item\label{it:sforbchwinf} a unique \chwm\ whose top space is \infdim;
		\item\label{it:sforbhwfin} at most two \hwms\ with \fdim\ top spaces.
	\end{enumerate}
\end{proposition}
\begin{proof}
	We start with some choice of highest weight $(j,\Delta) \in \CC^2$ and aim to show that the spectral flow orbit of $\bpmodjD$ has a \hwm\ with an \infdim\ top space.
	If $\bpmodjD$ already satisfies this requirement, then there is nothing to prove.
	So suppose that its top space is \fdim\ and let $v$ denote its \hwv.
	Then, $(G^-_0)^n v = 0$ for some minimal $n\ge1$.
	We set
	\begin{equation} \label{eq:deffg}
		\begin{aligned}
			f(j,\Delta) &= 3j^2 - (\kk+3)\Delta - (2\kk+3)j \\ \text{and} \quad
		g_n(j,\Delta) &= \frac{1}{n} \sum_{m=0}^{n-1} f(j-m,\Delta) = 3j^2 - (\kk+3)\Delta - (2\kk+3n)j + (n-1)(\kk+n+1),
		\end{aligned}
	\end{equation}
	so that \eqref{eq:commGG} gives
	\begin{equation}
		\begin{split}
			0
			&= \comm[\big]{G^+_0}{(G^-_0)^n} v
			= \sum_{m=0}^{n-1} (G^-_0)^{n-1-m} \comm[\big]{G^+_0}{G^-_0} (G^-_0)^m v
			= \sum_{m=0}^{n-1} (G^-_0)^{n-1-m} f(J_0,L_0) (G^-_0)^m v \\
			&= \sum_{m=0}^{n-1} f(j-m,\Delta) (G^-_0)^{n-1} v
			= n g_n(j,\Delta) (G^-_0)^{n-1} v,
		\end{split}
	\end{equation}
	hence $g_n(j,\Delta) = 0$.

	As $\bpmodjD$ is irreducible with \fdim\ top space, its image under the spectral flow functor $\bpsfsymb$ is also irreducible and \hw, with \hwv\ $\bpsfsymb \brac[\big]{(G^-_0)^{n-1} v}$.
	\cref{eq:bpsfweight} then gives
	\begin{equation} \label{eq:hwmsf1}
		\bpsfsymb(\bpmodjD) \cong \bpmod{j-n+1+\kappa}{\Delta+j-n+1}.
	\end{equation}
	If $\bpsfsymb(\bpmodjD)$ has an \infdim\ top space, then we are done.
	If not, then $g_m(j-n+1+\kappa,\Delta+j-n+1) = 0$ for some minimal $m\ge1$.
	However, this implies that
	\begin{equation}
		0 = g_m(j-n+1+\kappa,\Delta+j-n+1) - g_n(j,\Delta) = (3j+3-m-2n)(\kk+3-m-n).
	\end{equation}
	Noting that the last factor on the \rhs\ can only vanish if $\kk$ lies in $\ZZ_{\ge-1} \subset \frac{1}{2} \ZZ_{\ge-3}$, we conclude that
	\begin{equation} \label{eq:hwmsf2}
		h_{m,n}(j) = 3j+3-m-2n = 0.
	\end{equation}

	Continuing, $\bpsfsymb(\bpmodjD)$ having a \fdim\ top space means that $\bpsfsymb^2(\bpmodjD) \cong \bpsfsymb(\bpmod{j-n+1+\kappa}{\Delta+j-n+1})$ is another irreducible \hwm.
	If its top space were also \fdim, then we would conclude as above that $h_{\ell,m}(j-n+1+\kappa) = 0$ for some minimal $\ell\ge1$.
	However, this contradicts $\kk \notin \frac{1}{2} \ZZ_{\ge-3}$:
	\begin{equation} \label{eq:hwmsf3}
		0 = h_{\ell,m}(j-n+1+\kappa) - h_{m,n}(j) = 2(\kk+3)-\ell-m-n.
	\end{equation}
	This establishes the existence of a \hwm\ with \infdim\ top space in $\bpsforbit{\bpmodjD}$.

	We next claim that $\bpsforbit{\bpmodjD}$ also contains a \chwm\ with \infdim\ top space.
	This follows from the easily checked fact that applying $\bpsfsymb$ to an irreducible \chwm\ results in a \hwm:
	\begin{align*}
			&\quad \bpsforbit{\bpmodjD}\ \text{contains an irreducible \hwm} \\
			&\Ra \quad \bpconj \brac[\big]{\bpsforbit{\bpmodjD}}\ \text{contains an irreducible \chwm}\ \Mod{C} \\
			&\Ra \quad \Mod{H} = \bpsfsymb(\Mod{C})\ \text{is an irreducible \hwm\ in}\ \bpconj \brac[\big]{\bpsforbit{\bpmodjD}} \\
			&\Ra \quad \bpsforbit{\Mod{H}} = \bpconj \brac[\big]{\bpsforbit{\bpmodjD}}\ \text{contains an irreducible \hwm}\ \Mod{H}'\ \text{with \infdim\ top space} \\
			&\Ra \quad \bpconj(\Mod{H}')\ \text{is an irreducible \chwm\ with \infdim\ top space in}\ \bpsforbit{\bpmodjD}.
	\end{align*}

	Finally, the uniqueness of this \hw\ and \chwm\ in $\bpsforbit{\bpmodjD}$ follows from the fact that applying $\bpsfsymb^n$, $n>0$ ($n<0$), to a \hw\ $\ubpk$-module (\chw\ $\ubpk$-module) with \infdim\ top space results in a $\ubpk$-module that is not \rhw.
	This proves \ref{it:sforbhwinf} and \ref{it:sforbchwinf}, while \ref{it:sforbhwfin} now follows from the contradiction of \cref{eq:hwmsf3}.
\end{proof}

\begin{remark}
	Note that $\kk \in \frac{1}{2} \ZZ_{\ge-3}$ is equivalent to $\uu\ge2$ and $\vv=1$ or $2$.
	Moreover, for these $\uu$ and $\vv$, every irreducible \hw\ $\bpminmoduv$-module has a \fdim\ top space \cite{AraRat10,AdaCla19}.
	In particular, the spectral flow orbits never include modules with \infdim\ top spaces.
\end{remark}

It follows from \cref{prop:fdimhwm} that we will obtain a complete set of irreducible \hw\ $\ubpk$- or $\bpminmoduv$-modules with \fdim\ top spaces, the latter assuming $\uu\ge2$ and $\vv\ge3$, by looking at the spectral flow orbits of the irreducible \hwms\ with \infdim\ top spaces.
Indeed, it follows from the above analysis that if $\bpmodjD$ has an \infdim\ top space, then the only possible candidates for \fdim\ top spaces are $\bpsfsymb^{-1}(\bpmodjD)$ and $\bpsfsymb^{-2}(\bpmodjD)$.

We assemble the main results thus far, namely \cref{prop:zhuconsequences}\ref{it:isomorphic} as well as \cref{prop:irred=sf+relax,prop:completeuniv,prop:completeminmod,prop:infdimhwm,prop:fdimhwm}\ref{it:sforbhwfin}, as a \lcnamecref{thm:classification}.

\begin{theorem} \label{thm:classification}
	For $\kk\ne-3,-1,-\frac{3}{2}$ (coprime integers $\uu\ge2$ and $\vv\ge3$), every simple object of the category $\wcatk$ ($\wcatuv$) of generalised weight $\ubpk$-modules ($\bpminmoduv$-modules), with \fdim\ generalised weight spaces, is isomorphic to either:
	\begin{itemize}
		\item A spectral flow of an irreducible \frm\ $\bprmodjhw$ with $[j] \in \CC/\ZZ$ and $h,w \in \CC$ ($(h,w) \in \irrepsuv$).
		\item A spectral flow of an irreducible (\hw) quotient $\bpmodjD$ of a reducible \frm\ $\bprmod{j'}{h}{w}$ with $[j'] \in \CC/\ZZ$ and $h,w \in \CC$ ($(h,w) \in \irrepsuv$).
	\end{itemize}
\end{theorem}

\begin{remark} \label{rem:classbyhworchw}
	Considering \cref{prop:fdimhwm}\ref{it:sforbchwinf} instead of \ref{it:sforbhwfin} (or applying conjugation), it is clear that we can alternatively characterise the simple objects of $\wcatk$ and $\wcatuv$ as spectral flows of irreducible \frms\ and irreducible (\chw) submodules of reducible \frms.
\end{remark}

Algorithmically, this \lcnamecref{thm:classification} allows us to classify (subject to the stated restrictions on $\kk$, $\uu$ and $\vv$) the irreducible $\ubpk$- and $\bpminmoduv$-modules in $\wcatk$ and $\wcatuv$, respectively, using inverse \qhr:
\begin{itemize}
	\item For each $(h,w)$, determine for which $[j] \in \CC/\ZZ$, $\bprmodjhw = \lmod{j} \otimes \wmodhw$ is irreducible.
	\item For each of the (up to $3$) $[j] \in \CC/\ZZ$ with $\bprmodjhw$ reducible, identify its (unique) irreducible quotient $\bpmodjD$.
	\item Apply $\bpsfsymb^{\ell}$, for all $\ell \in \ZZ$, to all the irreducible $\bprmodjhw$ and $\bpmodjD$.
\end{itemize}
\noindent We shall see how to implement this algorithm with examples in the next \lcnamecref{sec:examples}.

\begin{remark}
	A natural question is whether inverse \qhr\ also allows one to analyse the possibility of nonsplit extensions between irreducible modules.
	For example, for nondegenerate levels, can one use the semisimplicity of the category of $\wminmoduv$-modules to prove the semisimplicity of the analogue of the BGG category $\categ{O}_{\kk}$ for $\bpminmoduv$?
	The latter fact was in fact established in \cite{FehCla20}, but by using minimal \qhr\ to relate it back to the semisimplicity \cite{AraRat12} of $\categ{O}_{\kk}$ for the simple affine \voa\ $\sslk$.
	However, we expect that this method will be difficult to generalise.
\end{remark}

\section{Examples} \label{sec:examples}

We apply the general results of the previous \cref{sec:reclass} to $\bpminmoduv$ for two classes of $(\uu,\vv)$.
The first, $\uu,\vv\ge3$, corresponds to $\kk$ being nondegenerate.
The second, $(\uu,\vv) = (2,3)$, corresponds to the nonadmissible level $\kk=-\frac{7}{3}$.

\subsection{Nondegenerate levels} \label{sec:nondeg}

In this \lcnamecref{sec:nondeg}, we classify irreducible \rhw\ $\bpminmoduv$-modules when $\uu,\vv\ge3$ ($\kk$ is nondegenerate).
This result was originally obtained in \cite{FehCla20} using properties of the minimal \qhr\ functor.
Here, we obtain it straightforwardly using inverse \qhr\ and lift it to a classification of the simple objects of $\wcatuv$, again when $\uu,\vv\ge3$.

Recall that for nondegenerate levels, $\irrepsuv$ is isomorphic, via the parametrisations $h_{[r,s]}$ and $w_{[r,s]}$ of \eqref{eq:hw(rs)}, to $(\pwlat{\uu-3} \times \pwlat{\vv-3}) / \ZZ_3$, where the $\ZZ_3$-action is effected by the permutation $\nabla$ of \eqref{eq:Z3action}.
We define
\begin{equation}
	j_{(r,s)} = \tfrac{1}{3} \brac[\big]{r_2-r_1-\tfrac{\uu}{\vv}(s_2-s_1-1)}, \quad
	(r,s) \in \pwlat{\uu-3} \times \pwlat{\vv-3},
\end{equation}
and recall that $[r,s] = \set[\big]{(r,s), \nabla(r,s), \nabla^2(r,s)}$.
\begin{theorem} \label{thm:nondegclass}
	Let $\kk$ be nondegenerate, so that $\uu,\vv\ge3$.
	Then, every irreducible $\bpminmoduv$-module in $\wcatuv$ is isomorphic to one, and only one, of the following:
	\begin{itemize}
		\item The $\bpsfsymb^{\ell} \brac[\big]{\bprmod{j}{h_{[r,s]}}{w_{[r,s]}}}$ with $\ell \in \ZZ$, $[r,s] \in (\pwlat{\uu-3} \times \pwlat{\vv-3})/\ZZ_3$ and $[j] \notin \set[\big]{[j_{(r',s')}] \st (r',s') \in [r,s]}$.
		\item The $\bpsfsymb^{\ell} \brac[\big]{\bpmod{j_{(r,s)}-1}{h_{[r,s]}+\kappa}}$ with $\ell \in \ZZ$ and $(r,s) \in \pwlat{\uu-3} \times \pwlat{\vv-3}$.
	\end{itemize}
\end{theorem}
\begin{proof}
	As $\bprmodjhw$ is almost irreducible with a top space possessing $1$-dimensional weight spaces (\cref{prop:zhuconsequences}\ref{it:1dimwtspaces}) and a bijective action of $G^+_0$ (\cref{thm:iqhrmod}\ref{it:G+inj}), it is reducible if and only if it has a \chwv\ in its top space.
	We test for such vectors by letting $G^-_0$ act, as per \eqref{eq:iqhr}, on the top space weight vector $\ee^{-b+(j+\kappa)c} \otimes v_{h,w}$.
	The result is
	\begin{equation}
		G^-_0 (\ee^{-b+(j+\kappa)c} \otimes v_{h,w}) = \brac[\big]{\alpha_{\kk} w + P_{\kk}(j,h)} \ee^{-b+(j-1+\kappa)c} \otimes v_{h,w},
	\end{equation}
	where $\alpha_{\kk}$ and $P_{\kk}$ were defined in \eqref{eq:actionofG-}.
	Substituting the parametrisations \eqref{eq:hw(rs)} and simplifying, we obtain
	\begin{equation} \label{eq:threesols}
		\alpha_{\kk} w_{[r,s]} + P_{\kk}(j,h_{[r,s]}) = -\ \prod_{\mathclap{(r',s') \in [r,s]}}\ (j-j_{(r',s')}),
	\end{equation}
	whence $\bprmod{j}{h_{[r,s]}}{w_{[r,s]}}$ is reducible if and only if $[j] = [j_{(r',s')}]$ for some $(r',s') \in [r,s]$.

	Fixing $[r,s] \in (\pwlat{\uu-3} \times \pwlat{\vv-3})/\ZZ_3$, hence $(h_{[r,s]},w_{[r,s]}) \in \irrepsuv$, it is easy to check that the three zeroes $j_{(r',s')}$, $(r',s') \in [r,s]$ of \eqref{eq:threesols} belong to distinct cosets in $\CC/\ZZ$.
	For example,
	\begin{equation}
		j_{\nabla(r,s)} - j_{(r,s)} = r_1+1 - \tfrac{\uu}{\vv} (s_1+1)
	\end{equation}
	is not an integer because $\uu$ and $\vv$ are coprime and $0 \le s_1 \le \vv-3$.
	We therefore have three distinct reducible \frms\ $\bprmod{j_{(r',s')}}{h_{[r,s]}}{w_{[r,s]}}$, $(r',s') \in [r,s]$, for each choice of $[r,s]$.
	Since $j_{(r',s')}$ is the weight of the \chwv\ in the top space, the irreducible quotient of $\bprmod{j_{(r',s')}}{h_{[r,s]}}{w_{[r,s]}}$ is isomorphic to the \hw\ $\bpminmoduv$-module $\bpmod{j_{(r',s')}-1}{h_{[r,s]}+\kappa}$, by \cref{prop:infdimhwm}\ref{it:infdimhw}.
	Moreover, the top space of the latter is clearly \infdim.
	The result now follows from \cref{thm:classification}.
\end{proof}

\begin{remark}
	For $\uu,\vv\ge3$ and $[r,s] \in (\pwlat{\uu-3} \times \pwlat{\vv-3})/\ZZ_3$, it is easy to see that the \chw\ submodule of $\bprmod{j_{(r',s')}}{h_{[r,s]}}{w_{[r,s]}}$ constructed in the proof of \cref{thm:nondegclass} is irreducible, hence isomorphic to $\bpcmod{j_{(r',s')}}{h_{[r,s]}+\kappa}$.
	It is also true, but less easy to see, that
	\begin{equation}
		\ses{\bpcmod{j_{(r',s')}}{h_{[r,s]}+\kappa}}{\bprmod{j_{(r',s')}}{h_{[r,s]}}{w_{[r,s]}}}{\bpmod{j_{(r',s')}-1}{h_{[r,s]}+\kappa}}
	\end{equation}
	is exact.
	This can be shown using an analogous argument to that of \cite[Sec.~4]{KawRel18}, see \cite[Thm.~4.24]{FehCla20}.
\end{remark}

This \lcnamecref{thm:nondegclass} then classifies the irreducible $\bpminmoduv$-modules in $\wcatuv$ when $\kk$ is nondegenerate.
One may of course continue the analysis, calculating how many \hwms\ with \fdim\ top spaces are in each spectral flow orbit and identifying their highest weights explicitly.
This is straightforward and we refer the interested reader to \cite[Sec.~2.3]{FehMod21}.

\subsection{Irreducible $\bpminmod{2}{3}$-modules} \label{sec:bpsinglet}

We turn to the classification of irreducible modules in $\wcat_{2,3}$.
The level $\kk=-\frac{7}{3}$, corresponding to $\uu=2$, $\vv=3$, $\kappa=-\frac{5}{9}$ and $\ccbp = -\frac{40}{3}$, is nonadmissible but may still be tackled using inverse \qhr, see \cref{thm:iqhr}\ref{it:iqhrminmod}.
What makes this an ideal case to study is that $\cczam = -2$ for this level and so the $\zamsymb$ minimal model $\wminmod{2}{3}$ coincides with Kausch's singlet algebra \cite{KauExt91}.

The irreducible \hw\ $\wminmod{2}{3}$-modules were classified by Wang in \cite{WanCla97}, see also \cite{EhoRep92,HonAut92,AdaCla02}.
Here, we review this classification following \cite[Sec.~3.3]{CreLog13}.
First, recall that $\wminmod{2}{3}$ is a vertex subalgebra of a rank-$1$ Heisenberg vertex algebra.
The latter's Fock spaces $\fock{\lambda}$, $\lambda \in \CC$, are thus $\wminmod{2}{3}$-modules by restriction.
A little calculation shows that the \hwv\ $v_{\lambda} \in \fock{\lambda}$ satisfies
\begin{equation} \label{eq:hw(lambda)}
	T_0 v_{\lambda} = h_{\lambda} v_{\lambda},\quad h_{\lambda} = \tfrac{1}{2} \lambda(\lambda+1), \qquad \text{and} \qquad
	W_0 v_{\lambda} = w_{\lambda} v_{\lambda},\quad w_{\lambda} = -\tfrac{1}{6\sqrt{2}} \lambda(\lambda+1)(2\lambda+1).
\end{equation}
The $\fock{\lambda}$ turn out to be irreducible, as $\wminmod{2}{3}$-modules, if and only if $\lambda \notin \ZZ$.
We therefore have the identification
\begin{equation}
	\fock{\lambda} \cong \wmod{h_{\lambda}}{w_{\lambda}}, \quad \lambda \notin \ZZ.
\end{equation}
These Fock spaces are sometimes referred to as the \emph{typical} irreducible $\wminmod{2}{3}$-modules.
The $\fock{\lambda}$ with $\lambda \in \RR$ are then the \emph{standard} $\wminmod{2}{3}$-modules, according to the standard module formalism of \cite{CreLog13,RidVer14}.

For $\lambda \in \ZZ$, $\fock{\lambda}$ has a unique irreducible submodule that we shall denote by $\sing{\lambda}$.
Moreover, the following short sequence is nonsplit and exact:
\begin{equation}
	\ses{\sing{\lambda}}{\fock{\lambda}}{\sing{\lambda+1}}.
\end{equation}
The $\sing{\lambda}$ are also \hw\ and we have
\begin{equation}
	\sing{\lambda} \cong
	\begin{cases}
		\wmod{h_{\lambda}}{w_{\lambda}}, & \lambda \in \ZZ_{\ge0}, \\
		\wmod{h_{\lambda-1}}{w_{\lambda-1}} & \lambda \in \ZZ_{<0}.
	\end{cases}
\end{equation}
These are then the \emph{atypical} irreducible $\wminmod{2}{3}$-modules.

It is easy to check from \eqref{eq:hw(lambda)} that the only nontrivial coincidence $(h_{\lambda},w_{\lambda}) = (h_{\mu},w_{\mu})$, $\lambda \ne \mu$, of highest weights occurs with $(h_0,w_0) = (0,0) = (h_{-1},w_{-1})$.
A complete set of mutually nonisomorphic irreducible \hw\ $\wminmod{2}{3}$-modules is thus given by the $\wmod{h_{\lambda}}{w_{\lambda}}$ with $\lambda \in \CC \setminus \set{-1}$.

\begin{theorem} \label{thm:bp(23)class}
	\leavevmode
	\begin{enumerate}
		\item\label{it:bp(23)fr} Every irreducible \fr\ $\bpminmod{2}{3}$-module is isomorphic to one, and only one, of the $\bprmod{j}{h_{\lambda}}{w_{\lambda}}$ with $\lambda \in \CC \setminus \set{-1}$ and $[j] \notin \set[\big]{[\frac{3\lambda+5}{9}], [\frac{3\lambda+2}{9}], [-\frac{6\lambda+1}{9}]}$.
		\item\label{it:bp(23)hw} Every irreducible \hw\ $\bpminmod{2}{3}$-module with an \infdim\ top space is isomorphic to one, and only one, of the following modules:
		\begin{enumerate}[label=(\textnormal{\roman*})]
			\item\label{it:bp(23)i} The $\bpmod{(3\lambda-4)/9}{h_{\lambda}-5/9}$ with $\lambda \in \CC \setminus \brac[\big]{\set{-1} \cup (\NN+\frac{1}{3})}$.
			\item\label{it:bp(23)ii} The $\bpmod{(3\lambda-7)/9}{h_{\lambda}-5/9}$ with $\lambda \in \CC \setminus \brac[\big]{\set{-1} \cup (\NN+\frac{2}{3})}$.
			\item\label{it:bp(23)iii} The $\bpmod{-(6\lambda+10)/9}{h_{\lambda}-5/9}$ with $\lambda \in \CC \setminus \brac[\big]{\set{-1} \cup (-\NN-\frac{1}{3}) \cup (-\NN-\frac{2}{3})}$.
		\end{enumerate}
		\item\label{it:bp(23)all} Every irreducible $\bpminmod{2}{3}$-module in $\wcat_{2,3}$ is isomorphic to a spectral flow of one, and only one, of these modules.
	\end{enumerate}
\end{theorem}
\begin{proof}
	We again look for \chwvs\ in the top space of $\bprmod{j}{h_{\lambda}}{w_{\lambda}}$, as in the proof of \cref{thm:nondegclass}.
	This time, the existence of such a vector is equivalent to the vanishing of
	\begin{equation} \label{eq:threesols'}
		\alpha_{-7/3} w_{\lambda} + P_{-7/3}(j,h_{\lambda}) = -\brac[\big]{j-\tfrac{3\lambda+5}{9}} \brac[\big]{j-\tfrac{3\lambda+2}{9}} \brac[\big]{j+\tfrac{6\lambda+1}{9}}.
	\end{equation}
	This determines when the \fr\ $\bpminmod{2}{3}$-module $\bprmod{j}{h_{\lambda}}{w_{\lambda}}$ is irreducible, proving \ref{it:bp(23)fr}.
	Note that the roots of \eqref{eq:threesols'} are the same for $\lambda = 0$ and $-1$.

	Unlike the nondegenerate case studied in \cref{thm:nondegclass}, the three zeroes of \eqref{eq:threesols'} need not belong to different cosets in $\CC/\ZZ$.
	Indeed, we have $[\frac{3\lambda+5}{9}] = [-\frac{6\lambda+1}{9}]$ for $\lambda \in \ZZ+\frac{1}{3}$ and $[\frac{3\lambda+2}{9}] = [-\frac{6\lambda+1}{9}]$ for $\lambda \in \ZZ-\frac{1}{3}$. %(of course $[\frac{3\lambda+5}{9}] \ne [\frac{3\lambda+2}{9}]$ for all $\lambda \in \CC$).
	For $\lambda \notin \ZZ\pm\frac{1}{3}$, it therefore follows that there are three irreducible \hw\ quotients, namely $\bpmod{(3\lambda-4)/9}{h_{\lambda}-5/9}$, $\bpmod{(3\lambda-7)/9}{h_{\lambda}-5/9}$ and $\bpmod{-(6\lambda+10)/9}{h_{\lambda}-5/9}$, and that each has an \infdim\ top space.

	Suppose now that $\lambda \in \ZZ+\frac{1}{3}$.
	Then, there is a single zero of \eqref{eq:threesols'} in $[\frac{3\lambda+2}{9}]$ and so $\bpmod{(3\lambda-7)/9}{h_{\lambda}-5/9}$ is the irreducible \hw\ quotient (with \infdim\ top space).
	However, there are two zeroes in $[\frac{3\lambda+5}{9}] = [-\frac{6\lambda+1}{9}]$, hence two \chwvs\ in the top space of $\bprmod{j}{h_{\lambda}}{w_{\lambda}}$.
	In other words, $\bprmod{j}{h_{\lambda}}{w_{\lambda}}$ has two \chw\ submodules, one of which contains the other.
	We want the quotient by the larger of the two, which is the one whose \chwv\ has the smallest $J_0$-eigenvalue.
	If $\lambda<0$, then this eigenvalue is $\frac{3\lambda+5}{9}$, hence $\bpmod{(3\lambda-4)/9}{h_{\lambda}-5/9}$ is the irreducible \hw\ quotient (with \infdim\ top space).
	Otherwise, it is $-\frac{6\lambda+1}{9}$ and the desired quotient is $\bpmod{-(6\lambda+10)/9}{h_{\lambda}-5/9}$.

	The analysis for $\lambda \in \ZZ-\frac{1}{3}$ is very similar.
	To complete the proof of \ref{it:bp(23)hw}, we only have to check that the members of the three \hw\ families are all distinct.
	This is easily verified.
	For example, $\brac[\big]{\frac{3\lambda-4}{9},h_{\lambda}-\frac{5}{9}} = \brac[\big]{-\frac{6\mu+10}{9},h_{\mu}-\frac{5}{9}}$ gives two solutions: $\lambda = 0$, $\mu = -1$; and $\lambda = \mu = -\frac{2}{3}$.
	In both cases, $\lambda$ corresponds to a family member but $\mu$ does not.

	Finally, \ref{it:bp(23)all} now follows from \ref{it:bp(23)fr}, \ref{it:bp(23)hw} and \cref{thm:classification}.
\end{proof}

\begin{remark} \label{rem:infdimcoincidences}
	The exclusions for the parameter $\lambda$ in the families of \cref{thm:bp(23)class}\ref{it:bp(23)hw} avoid the following coincidences:
	\begin{itemize}
		\item $\bpmod{-7/9}{-5/9}$ belongs to family \ref{it:bp(23)i} with $\lambda = -1$ and family \ref{it:bp(23)ii} with $\lambda = 0$.
		\item $\bpmod{-10/9}{-5/9}$ belongs to family \ref{it:bp(23)ii} with $\lambda = -1$ and family \ref{it:bp(23)iii} with $\lambda = 0$.
		\item $\bpmod{-4/9}{-5/9}$ belongs to family \ref{it:bp(23)iii} with $\lambda = -1$ and family \ref{it:bp(23)i} with $\lambda = 0$.
		\item $\bpmod{-8/9}{-2/3}$ belongs to family \ref{it:bp(23)iii} with $\lambda = -\frac{1}{3}$ and family \ref{it:bp(23)ii} with $\lambda = -\frac{1}{3}$.
		\item $\bpmod{-2/3}{-2/3}$ belongs to family \ref{it:bp(23)iii} with $\lambda = -\frac{2}{3}$ and family \ref{it:bp(23)i} with $\lambda = -\frac{2}{3}$.
	\end{itemize}
\end{remark}

\begin{remark} \label{rem:O23nonss}
	The proof of \cref{thm:bp(23)class} shows that there exist reducible \chw\ $\bpminmod{2}{3}$-modules.
	Conjugating therefore gives the same conclusion in the \hw\ case.
	The analogue of the BGG category $\categ{O}_{\kk}$ for $\bpminmod{2}{3}$ is consequently nonsemisimple.
\end{remark}

\begin{conjecture}
	The analogue of the BGG category $\categ{O}_{\kk}$ for $\bpminmoduv$ is semisimple if and only if $\uu=2$ and $\vv=1$, $\uu\ge3$ and $\vv=2$, or $\uu,\vv\ge3$.
\end{conjecture}

While \cref{thm:bp(23)class} classifies the irreducibles in $\wcat_{2,3}$, it may be made more explicit by determining those $(j,\Delta)$ for which $\bpmodjD$ is an irreducible \hw\ $\bpminmod{2}{3}$-module with a \fdim\ top space.
These are precisely the weight modules whose $L_0$-eigenvalues are bounded below and whose $L_0$-eigenspaces are \fdim, that is they are ordinary modules.

\begin{theorem} \label{thm:bp(23)fdtopclass}
	Every irreducible ordinary $\bpminmod{2}{3}$-module is isomorphic to one, and only one, of the following:
	\begin{enumerate}
		\item\label{it:bp(23)1} The $\bpmod{\lambda/3}{h_{\lambda}+\lambda/3}$ with $\lambda \in \CC \setminus \set{-\frac{5}{3}}$ and top space dimension $1$.
		\item\label{it:bp(23)2} The $\bpmod{(3\lambda-4)/9}{h_{\lambda}-5/9}$ with $\lambda \in \NN+\frac{4}{3}$ and top space dimension $\lambda+\frac{2}{3}$.
		\item\label{it:bp(23)3} The $\bpmod{(3\lambda-7)/9}{h_{\lambda}-5/9}$ with $\lambda \in \NN+\frac{2}{3}$ and top space dimension $\lambda+\frac{1}{3}$.
		\item\label{it:bp(23)4} The $\bpmod{-(6\lambda+10)/9}{h_{\lambda}-5/9}$ with $\lambda \in -\NN-\frac{8}{3}$ and top space dimension $-\lambda-\frac{2}{3}$.
		\item\label{it:bp(23)5} The $\bpmod{-(6\lambda+10)/9}{h_{\lambda}-5/9}$ with $\lambda \in -\NN-\frac{7}{3}$ and top space dimension $-\lambda-\frac{1}{3}$.
	\end{enumerate}
\end{theorem}
\begin{proof}
	Suppose that the irreducible \hw\ $\bpminmod{2}{3}$-module $\bpmodjD$ has a top space of dimension $n$.
	By \cref{eq:deffg}, this is equivalent to $n$ being the smallest positive integer satisfying $g_n(j,\Delta) = 0$.
	Moreover, either $\bpsfsymb(\bpmodjD)$ or $\bpsfsymb^2(\bpmodjD)$ is \hw\ with an \infdim\ top space, by \cref{prop:fdimhwm}.

	Suppose that it is $\bpsfsymb(\bpmodjD)$.
	Then, we recall that $\bpsfsymb(\bpmodjD) \cong \bpmod{j-n+4/9}{\Delta+j-n+1}$, by \eqref{eq:hwmsf1}, and compare with the classification in \cref{thm:bp(23)class}\ref{it:bp(23)hw}.
	There are thus three possibilities:
	\begin{itemize}
		\item $j-n+\frac{4}{9} = \frac{3\lambda-4}{9}$ and $\Delta+j-n+1 = h_{\lambda} - \frac{5}{9}$ for some $\lambda \notin \set{-1} \cup (\NN+\frac{1}{3})$.
		In this case, solving for $j$ and $\Delta$ results in $g_n(j,\Delta) = (n-\frac{1}{3})(n+\lambda)$.
		As $n$ must be a positive integer, this only vanishes when $\lambda = -n \in \ZZ_{\le-1}$.
		However, $\lambda=-1$ is explicitly excluded, so we only take $\lambda \in \ZZ_{\le-2}$.
		Substituting back, our solution becomes $\bpmodjD = \bpmod{-2(3\lambda+4)/9}{(\lambda-1)(3\lambda+4)/6}$.
		If we set $\lambda = \mu+\frac{1}{3}$, we recognise the family~\ref{it:bp(23)5} module $\bpmodjD = \bpmod{-(6\mu+10)/9}{h_{\mu}-5/9}$.
		Its top space has dimension $n = -\lambda = -\mu-\frac{1}{3}$, where $\mu \in -\NN-\frac{7}{3}$.
		\item $j-n+\frac{4}{9} = \frac{3\lambda-7}{9}$ and $\Delta+j-n+1 = h_{\lambda} - \frac{5}{9}$ for some $\lambda \notin \set{-1} \cup (\NN+\frac{2}{3})$.
		Following the same steps as in the previous case now gives $g_n(j,\Delta) = (n-1)(n+\lambda-\frac{1}{3})$.
		The smallest positive-integer solution for $n$ is therefore always $1$, so $\bpmodjD = \bpmod{(3\lambda-2)/9}{(\lambda+1)(3\lambda-2)/6}$ has a $1$-dimensional top space.
		Setting $\mu = \lambda - \frac{2}{3}$, we recognise these modules as belonging to family~\ref{it:bp(23)1} with $\mu \notin \set{-\frac{5}{3}} \cup \NN$.
		\item $j-n+\frac{4}{9} = -\frac{6\lambda+10}{9}$ and $\Delta+j-n+1 = h_{\lambda} - \frac{5}{9}$ for some $\lambda \notin \set{-1} \cup (-\NN-\frac{1}{3}) \cup (-\NN-\frac{2}{3})$.
		This time, we get $g_n(j,\Delta) = (n-\lambda-1) (n-\lambda-\frac{4}{3})$, hence two distinct families of solutions: $\lambda = n-1 \in \NN$ and $\lambda = n-\frac{4}{3} \in \NN+\frac{2}{3}$ (we have to exclude $\lambda = -\frac{1}{3}$).
		For $\lambda \in \NN$, we set $\mu = \lambda+\frac{2}{3}$ to recognise the $\bpmodjD$ as belonging to family~\ref{it:bp(23)3} with $\mu \in \NN+\frac{2}{3}$.
		For $\lambda \in \NN+\frac{2}{3}$, $\mu = \lambda+\frac{2}{3}$ instead results in the $\bpmodjD$ belonging to family~\ref{it:bp(23)2} with $\mu \in \NN+\frac{4}{3}$.
	\end{itemize}

	The only alternative is that $\bpsfsymb^2(\bpmodjD)$ is \hw\ with an \infdim\ top space.
	Then, $\bpsfsymb(\bpmodjD)$ must belong to one of the four families of \hw\ $\bpminmod{2}{3}$-modules with \fdim\ top spaces that we have already discovered.
	The analysis for families~\ref{it:bp(23)2}, \ref{it:bp(23)3} and \ref{it:bp(23)5} is then the same as above, except that $\lambda$ is now required to lie in $\NN+\frac{4}{3}$, $\NN+\frac{2}{3}$ and $-\NN-\frac{7}{3}$, respectively.
	The results are that this is impossible when $\bpsfsymb(\bpmodjD)$ belongs to families~\ref{it:bp(23)2} and \ref{it:bp(23)5}, but for family~\ref{it:bp(23)3} the $\bpmodjD$ are found to belong to family~\ref{it:bp(23)1} with $\lambda \in \NN$.

	It only remains to consider if $\bpsfsymb(\bpmodjD)$ can belong to family~\ref{it:bp(23)1} (with $\lambda \notin \set{-\frac{5}{3}} \cup \NN$).
	Setting $j-n+\frac{4}{9} = \frac{\lambda}{3}$ and $\Delta+j-n+1 = h_{\lambda} + \frac{\lambda}{3}$, we deduce that $g_n(j,\Delta) = (n+\frac{1}{3})(n+\lambda+\frac{2}{3})$.  Noting that $\lambda = -n-\frac{2}{3} \in -\NN-\frac{8}{3}$, because $-\frac{5}{3}$ must be excluded, we conclude that $\bpmodjD$ belongs to family~\ref{it:bp(23)4}.
\end{proof}

\begin{corollary} \label{cor:bp23hw}
	Every irreducible \hw\ $\bpminmod{2}{3}$-module is isomorphic to one in the set
	\begin{equation} \label{eq:bp23hwset}
		\set[\big]{\bpmod{(3\lambda-4)/9}{h_{\lambda}-5/9}, \bpmod{(3\lambda-7)/9}{h_{\lambda}-5/9},
		           \bpmod{-(6\lambda+10)/9}{h_{\lambda}-5/9}, \bpmod{\lambda/3}{h_{\lambda}+\lambda/3}
		           \st \lambda \in \CC}.
	\end{equation}
\end{corollary}
\noindent
We may equivalently reparametrise the four families in \eqref{eq:bp23hwset} using the $J_0$-eigenvalue of the \hwv:
\begin{equation} \label{eq:bp23hwset'}
	\set[\big]{\bpmod{j}{(j+1)(9j+2)/2}, \bpmod{j}{(3j+4)(9j+5)/6}, \bpmod{j}{j(9j+14)/8}, \bpmod{j}{j(9j+5)/2} \st j \in \CC}.
\end{equation}
For convenience, (a part of) this set is plotted (for real $j$) in \cref{fig:bp23-plot}.

\begin{figure}
	\includegraphics[width=0.6\textwidth]{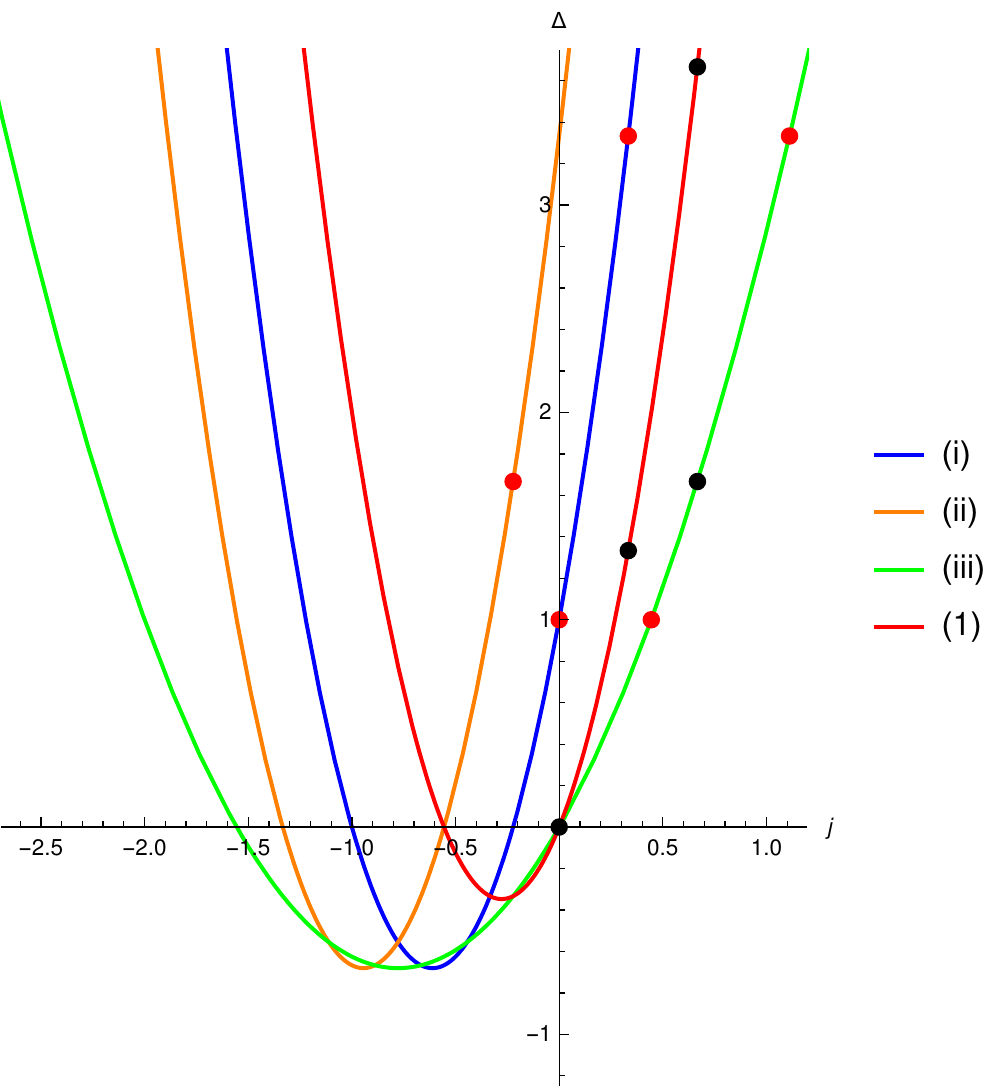}
\caption{%
	A picture of the highest weights $(j,\Delta)$ appearing in the set \eqref{eq:bp23hwset'}, with $j$ real.
	The blue, orange and green curves indicate the families~\ref{it:bp(23)i}, \ref{it:bp(23)ii} and \ref{it:bp(23)iii} of \hw\ $\bpminmod{2}{3}$-modules described in \cref{thm:bp(23)class}.
	These modules generically have \infdim\ top spaces.
	The red curve indicates family~\ref{it:bp(23)1} in \cref{thm:bp(23)fdtopclass}.
	Its modules have \fdim\ top spaces and their images under $\bpsfsymb$ have \infdim\ top spaces.
	A red dot indicates another module with these properties (families~\ref{it:bp(23)2}, \ref{it:bp(23)3} and \ref{it:bp(23)5} in \cref{thm:bp(23)fdtopclass}).
	On the other hand, a black dot indicates a module with a \fdim\ top space whose image under $\bpsfsymb$ also has a \fdim\ top space (family~\ref{it:bp(23)1} with $j \in \frac{1}{3} \NN$ and family~\ref{it:bp(23)4}).
} \label{fig:bp23-plot}
\end{figure}

\begin{remark} \label{rem:fdimcoincidences}
	The exclusions for the parameter $\lambda$ in the families of \cref{thm:bp(23)fdtopclass} avoid the following coincidences:
	\begin{itemize}
		\item $\bpmod{-1/3}{-1/3}$ belongs to family~\ref{it:bp(23)1} with $\lambda = -1$ and family~\ref{it:bp(23)2} with $\lambda = \frac{1}{3}$.
		\item $\bpmod{-5/9}{0}$ belongs to family~\ref{it:bp(23)1} with $\lambda = -\frac{5}{3}$ and family~\ref{it:bp(23)3} with $\lambda = \frac{2}{3}$.
		\item $\bpmod{0}{0}$ belongs to family~\ref{it:bp(23)1} with $\lambda = 0$ and family~\ref{it:bp(23)4} with $\lambda = -\frac{5}{3}$.
		\item $\bpmod{-2/9}{-1/3}$ belongs to family~\ref{it:bp(23)1} with $\lambda = -\frac{2}{3}$ and family~\ref{it:bp(23)5} with $\lambda = -\frac{4}{3}$.
	\end{itemize}
\end{remark}

\begin{remark} \label{rem:bp23sforbits}
	We record for completeness the result of applying spectral flow to a \hw\ $\bpminmod{2}{3}$-module with \fdim\ top space:
	\begin{itemize}
		\item For $\lambda \in \CC$, $\bpsfsymb(\bpmod{\lambda/3}{h_{\lambda}+\lambda/3}) \cong \bpmod{(3\mu-7)/9}{h_{\mu}-5/9}$, where $\mu = \lambda + \frac{2}{3}$.
		\item For $\lambda \in \NN+\frac{1}{3}$, $\bpsfsymb(\bpmod{(3\lambda-4)/9}{h_{\lambda}-5/9}) \cong \bpmod{-(6\mu+10)/9}{h_{\mu}-5/9}$, where $\mu = \lambda - \frac{2}{3}$.
		\item For $\lambda \in \NN+\frac{2}{3}$, $\bpsfsymb(\bpmod{(3\lambda-7)/9}{h_{\lambda}-5/9}) \cong \bpmod{-(6\mu+10)/9}{h_{\mu}-5/9}$, where $\mu = \lambda - \frac{2}{3}$.
		\item For $\lambda \in -\NN-\frac{5}{3}$, $\bpsfsymb(\bpmod{-(6\lambda+10)/9}{h_{\lambda}-5/9}) \cong \bpmod{\mu/3}{h_{\mu}+\mu/3}$, where $\mu = \lambda$.
		\item For $\lambda \in -\NN-\frac{4}{3}$, $\bpsfsymb(\bpmod{-(6\lambda+10)/9}{h_{\lambda}-5/9}) \cong \bpmod{(3\mu-4)/9}{h_{\mu}-5/9}$, where $\mu = \lambda + \frac{1}{3}$.
	\end{itemize}
	We can thus roughly summarise the corresponding spectral flow orbits in terms of the families of \cref{thm:bp(23)class,thm:bp(23)fdtopclass} as follows (ignoring modules that are not \hw):
	\begin{equation}
		\ref{it:bp(23)1} \xrightarrow{\bpsfsymb} \ref{it:bp(23)3} \xrightarrow{\bpsfsymb} \ref{it:bp(23)iii}, \qquad
		\ref{it:bp(23)2} \xrightarrow{\bpsfsymb} \ref{it:bp(23)iii}, \qquad
		\ref{it:bp(23)4} \xrightarrow{\bpsfsymb} \ref{it:bp(23)1} \xrightarrow{\bpsfsymb} \ref{it:bp(23)ii}, \qquad
		\ref{it:bp(23)5} \xrightarrow{\bpsfsymb} \ref{it:bp(23)i}.
	\end{equation}
	Of course, there are also an uncountably infinite number of orbits with a single \hwm.
\end{remark}

\section{An application to $\slthree$ minimal models} \label{sec:sl3}

We finish by studying some of the implications of our results, when combined with other known relationships, to $\slthree$ minimal models.
We denote the universal level-$\kk$ affine \voa\ associated with $\slthree$ by $\uslk$ and its simple quotient by $\sslk$.
When $\kk$ is expressed in terms of $\uu$ and $\vv$, as in \eqref{eq:kuv}, we shall also write $\sslk = \slminmoduv$ and refer to the latter as an $\slthree$ minimal model \voa.

Recall that $\ubpk$ is the \qhr\ of $\uslk$ corresponding to the minimal (and subregular) nilpotent orbit \cite{PolGau90,BerCon91,KacQua03}.
We restrict the corresponding reduction functor $\qhrmin$ to the \kl\ category $\klcat^{\kk}$ of ordinary $\uslk$-modules, these being the weight modules with bounded-below $L_0$-eigenvalues and \fdim\ $L_0$-eigenspaces.
The simple objects of $\klcat^{\kk}$ are thus the irreducible \hwms\ whose highest weights have the form $(\kk-r-s+2)\fwt{0} + (r-1) \fwt{1} + (s-1) \fwt{2}$, for some $r,s \in \ZZ_{\ge1}$.
Here, $\fwt{i}$, $i=0,1,2$, denotes the $i$-th fundamental weight of $\aslthree$.
We denote the irreducible \hw\ $\uslk$-module of this highest weight by $\slmod{r}{s}$.

\begin{proposition} \label{prop:red1}
	For $\kk \notin \set{-3} \cup \ZZ_{\ge-1}$ and $r,s \in \ZZ_{\ge1}$, the minimal \qhr\ of $\slmod{r}{s}$ is the irreducible \hw\ $\ubpk$-module $\qhrmin(\slmod{r}{s}) = \bpmod{j_{r,s}}{\Delta_{r,s}}$, where
	\begin{equation} \label{eq:jDrs}
		j_{r,s} = \frac{r+2s-3}{3} \quad \text{and} \quad
		\Delta_{r,s} = \frac{r^2+rs+s^2-3}{3(\kk+3)} - \frac{2r+s-3}{3}.
	\end{equation}
	Moreover, the top space of $\bpmod{j_{r,s}}{\Delta_{r,s}}$ has dimension $s$.
\end{proposition}
\begin{proof}
	Since $\kk \notin \NN$, the zeroth Dynkin label of the highest weight of $\slmod{r}{s}$ is not in $\NN$.
	The minimal reduction of $\slmod{r}{s}$ is thus an irreducible \hwm, by \cite[Thm.~6.7.4]{AraRep04}.
	Moreover, its highest weight corresponds to the quoted formulae for $j_{r,s}$ and $\Delta_{r,s}$, by \cite[Thm.~6.3]{KacQua03b}.
	It remains to check that its top space has dimension $s$.
	This follows from \eqref{eq:deffg} and $\kk \notin \ZZ_{\ge-1}$, because
	\begin{equation}
		g_n(j_{r,s},\Delta_{r,s}) = (n-r-s+\kk+3)(n-s). \qedhere
	\end{equation}
\end{proof}

\begin{proposition} \label{prop:red2}
	For $\kk \notin \set{-3} \cup \ZZ_{\ge-1}$ and $r,s \in \ZZ_{\ge1}$, the (irreducible) $\ubpk$-module $\bpmod{j_{r,s}}{\Delta_{r,s}}$ may be realised as a submodule of $\bpsfsymb(\bprmod{j'_{r,s}}{h_{r,s}}{w_{r,s}})$, where
	\begin{equation} \label{eq:j'hwrs}
		\begin{gathered}
			j'_{r,s} = \frac{r+2s-2(\kk+3)}{3}, \quad
			h_{r,s} = \frac{r^2+rs+s^2-3}{3(\kk+3)}-r-s+2 \\ \text{and} \quad
			w_{r,s} = -\frac{\sqrt{3}}{(\kk+3)^3/2} \frac{r-s}{3} \brac*{\frac{2r+s}{3}-\kk-3} \brac*{\frac{r+2s}{3}-\kk-3}.
		\end{gathered}
	\end{equation}
\end{proposition}
\begin{proof}
	As $\bpmod{j_{r,s}}{\Delta_{r,s}}$ is \hw, with a \fdim\ top space, it is isomorphic to either $\bpsfsymb(\Mod{C})$ or $\bpsfsymb^2(\Mod{C})$, where $\Mod{C}$ is a \chwm\ with an \infdim\ top space (\cref{prop:fdimhwm}).
	Suppose that it is $\bpsfsymb(\Mod{C})$.
	Then, $\bpmod{j_{r,s}}{\Delta_{r,s}}$ is isomorphic to a submodule of $\bpsfsymb(\bprmod{j'_{r,s}}{h_{r,s}}{w_{r,s}})$, for some $[j'_{r,s}] \in \CC/\ZZ$ and $h_{r,s},w_{r,s} \in \CC$, by \cref{prop:infdimhwm}\ref{it:infdimchw}.
	The \hwv\ of $\bpmod{j_{r,s}}{\Delta_{r,s}}$ is mapped to the \chwv\ of $\Mod{C}$ by $\bpsfsymb^{-1}$ and the weight of the latter is $(j_{r,s}-\kappa, \Delta_{r,s}-j_{r,s}+\kappa)$, by \eqref{eq:bpsfweight}.
	As in the proof of \cref{prop:infdimhwm}, this identifies $[j'_{r,s}] = [j_{r,s}-\kappa]$ and $h_{r,s} = \Delta_{r,s}-j_{r,s}$.
	To obtain $w_{r,s}$, substitute $j'_{r,s}$ and $h_{r,s}$ into \eqref{eq:actionofG-}.
	As we have found a solution, there is no need to consider the possibility that $\bpmod{j_{r,s}}{\Delta_{r,s}} \cong \bpsfsymb^2(\Mod{C})$.
\end{proof}

\begin{remark} \label{rem:ordinaryrealisation}
	\cref{prop:red2} constructs an embedding $\bpmod{j_{r,s}}{\Delta_{r,s}} \ira \bpsfsymb(\bprmod{j'_{r,s}}{h_{r,s}}{w_{r,s}}) = \lsfsymb(\lmod{j'_{r,s}}) \otimes \wmod{h_{r,s}}{w_{r,s}}$.
	However, $\lmod{j'_{r,s}} \cong \lvoa \, \ee^{-b+(j'_{r,s}+\kappa)c} = \lvoa \, \ee^{-b+(r+2s-3)c/3}$ and thus
	\begin{equation}
		\lsfsymb(\lmod{j'_{r,s}}) = \lvoa \, \ee^{(r+2s-3)c/3} \in \lvoa^{1/3}, \quad
		\lvoa^{1/3} = \lvoa \oplus \lvoa \, \ee^{c/3} \oplus \lvoa \, \ee^{2c/3},
	\end{equation}
	by \eqref{eq:lattsfvec}.
	It follows that this \cref{prop:red2} constructs the ordinary $\ubpk$-modules $\bpmod{j_{r,s}}{\Delta_{r,s}}$ as submodules of $\lvoa^{1/3} \otimes \wmod{h_{r,s}}{w_{r,s}}$.
	This is thus the analogue of the realisation of ordinary $\uaff{\kk}{\sltwo}$-modules presented in \cite[Sec.~6]{AdaRea17}.
\end{remark}

We have the following important consequence.
\begin{theorem} \label{thm:sl3bpw3}
	Assume that $\uu\ge2$ and $\vv\ge3$ are coprime and that $r, s \in \ZZ_{\ge1}$.
	Let $j_{r,s}$, $\Delta_{r,s}$, $h_{r,s}$ and $w_{r,s}$ be defined by \eqref{eq:jDrs} and \eqref{eq:j'hwrs}.
	Then, the following conditions are equivalent:
	\begin{enumerate}
		\item\label{it:equivsl3} $\slmod{r}{s}$ is an $\slminmoduv$-module.
		\item\label{it:equivbp} $\bpmod{j_{r,s}}{\Delta_{r,s}}$ is a $\bpminmoduv$-module.
		\item\label{it:equivw3} $\wmod{h_{r,s}}{w_{r,s}}$ is a $\wminmoduv$-module.
	\end{enumerate}
\end{theorem}
\begin{proof}
	\ref{it:equivsl3} $\Ra$ \ref{it:equivbp} is a standard result about \qhr, see for example \cite[Prop.~4.7]{FehCla20}.

	For~\ref{it:equivbp} $\Ra$ \ref{it:equivsl3}, there also exists an inverse reduction embedding \cite[Thm.~5.2]{AdaRel21}
	\begin{equation}
		\slminmoduv \ira \bpminmoduv \otimes \VOA{SB} \otimes \lvoa,
	\end{equation}
	where $\VOA{SB}$ denotes the symplectic bosons \voa\ (also known as bosonic ghosts).
	Moreover, calculation shows that $\slmod{r}{s}$ may be explicitly realised \cite[Thm.~6.3(2)]{AdaRel21} a submodule of the tensor product of $\bpmod{j_{r,s}}{\Delta_{r,s}}$, $\VOA{SB}$ and a direct summand of $\lvoa^{1/3}$.

	So far, the proven implications hold for $\uu,\vv\ge2$.
	For~\ref{it:equivbp} $\Lra$ \ref{it:equivw3}, note that \cref{prop:red2} shows that $\bpsfsymb^{-1}(\bpmod{j_{r,s}}{\Delta_{r,s}})$ is an irreducible \chw\ submodule of a \frm.
	By \cref{thm:classification,rem:classbyhworchw}, which require $\vv\ge3$, $\bpsfsymb^{-1}(\bpmod{j_{r,s}}{\Delta_{r,s}})$ is a $\bpminmoduv$-module if and only if this \frm\ is.
	But, the latter condition is equivalent to $(h_{r,s},w_{r,s}) \in \irrepsuv$.
\end{proof}

When $\kk$ is nondegenerate ($\uu,\vv\ge3$), \ref{it:equivsl3} $\Lra$ \ref{it:equivbp} is exactly \cite[Thm.~4.8]{FehCla20}.
For $\uu=2$, we believe that this equivalence is new.
Here is an interesting \lcnamecref{cor:a23kl} for the \kl\ category $\klcat_{2,3}$ of ordinary $\slminmod{2}{3}$-modules.
\begin{corollary} \label{cor:a23kl}
	Every simple object in $\klcat_{2,3}$ is isomorphic to a module from the set
	\begin{equation}
		\set[\big]{\slmod{n}{1}, \slmod{1}{n} \st n \in \ZZ_{\ge1}}.
	\end{equation}
\end{corollary}
\begin{proof}
	This follows from \cref{thm:sl3bpw3} by comparing the formulae in \eqref{eq:jDrs} with the classification of irreducible ordinary $\bpminmod{2}{3}$-modules in \cref{thm:bp(23)fdtopclass}.
	The result is that the only solutions with $r,s \in \ZZ_{\ge1}$ correspond to families~\ref{it:bp(23)1} and \ref{it:bp(23)4} of the latter \lcnamecref{thm:bp(23)fdtopclass}, the former with $\lambda \in \NN$, $r=\lambda+1$, $s=1$ and the latter with $\lambda \in -\NN-\frac{5}{3}$, $r=1$, $s=-\lambda-\frac{2}{3}$.
\end{proof}

\begin{remark}
	Note that the two families of irreducible ordinary $\bpminmod{2}{3}$-modules that arise as minimal \qhr s of the irreducible ordinary $\slminmod{2}{3}$-modules are precisely those whose images under $\bpsfsymb$ are again ordinary.
	Indeed, for $n \in \ZZ_{\ge1}$, \cref{rem:bp23sforbits,prop:red1} give
	\begin{equation}
		\begin{gathered}
			\slmod{n}{1} \xrightarrow{\qhrmin} \bpmod{(n-1)/3}{h_{n-1}+(n-1)/3} \xrightarrow{\ \bpsfsymb\ } \bpmod{(3\lambda-7)/9}{h_{\lambda}-5/9}, \quad \text{where}\ \lambda = n-\tfrac{1}{3}, \\
			\slmod{1}{n} \xrightarrow{\qhrmin} \bpmod{-(6\lambda+10)/9}{h_{\lambda}-5/9} \xrightarrow{\ \bpsfsymb\ } \bpmod{\lambda/3}{h_{\lambda}+\lambda/3}, \quad \text{where}\ \lambda = -n-\tfrac{2}{3}.
		\end{gathered}
	\end{equation}
\end{remark}

\begin{remark}
	We believe that $\klcat_{2,3}$ is semisimple.
	We will study this category in forthcoming publications.
\end{remark}

\flushleft
%\bibliography{bp}
%\bibliographystyle{noc_plain}
\providecommand{\opp}[2]{\textsf{arXiv:\mbox{#2}/#1}}
\providecommand{\pp}[2]{\textsf{arXiv:#1 [\mbox{#2}]}}

\end{document}